\documentclass[11pt,twoside]{amsart}
\usepackage{t1enc}
\usepackage[latin2]{inputenc}

\usepackage{amsmath,amsfonts}
\usepackage{amsthm}
\usepackage{amscd}
\usepackage{amssymb}
\usepackage{enumerate}
\usepackage{mathrsfs}
\usepackage{dsfont}

\usepackage{stmaryrd}
\usepackage[T1]{fontenc}

\usepackage{mathtools}

\usepackage{bbm}

\usepackage{tikz}
\usepackage{tikz-cd}
\usepackage{wasysym}
\usepackage{verbatim}
\usepackage{fancyhdr}
\usepackage{fancybox}
\usepackage{url}
\usepackage{color}

\newcommand{\eps}{\varepsilon}

\newcommand{\bc}{\begin{center}}
\newcommand{\ec}{\end{center}}

\newtheorem{thm}{Theorem}[section]
\newtheorem{lem}[thm]{Lemma}
\newtheorem{prop}[thm]{Proposition}
\newtheorem{cor}[thm]{Corollary}
\newtheorem{fact}[thm]{Fact}

\theoremstyle{definition}
\newtheorem{que}[thm]{Question}
\newtheorem{prob}[thm]{Problem}

\newtheorem{exa}[thm]{Example}

\newtheorem{rem}[thm]{Remark}

\title[On sequences of homomorphisms into measure algebras]{On sequences of homomorphisms into measure algebras and the Efimov problem}

\author{Piotr Borodulin--Nadzieja}
\address[Piotr Borodulin-Nadzieja]{Instytut Matematyczny, Uniwersytet Wroc\l awski, \ \ \  pl. Grunwaldzki 2/4, 50-384 Wroc\l aw, Poland}
\email{pborod@math.uni.wroc.pl}

\author{Damian Sobota}
\address[Damian Sobota]{Kurt G\"odel Research Center for Mathematical Logic, Departament of Mathematics, University of Vienna, Kolingasse 14-16, 1090 Vienna, Austria}
\email{damian.sobota@univie.ac.at}

\thanks{The first author was supported by 
 National Science Center project no. 2018/29/B/ST1/00223. The second author was supported by the Austrian Science Fund, grant M2500-N35.}


\begin{document}

\maketitle

\begin{abstract}
For given Boolean algebras $\mathbb{A}$ and $\mathbb{B}$ we endow the space $\mathcal{H}(\mathbb{A},\mathbb{B})$ of all Boolean homomorphisms from $\mathbb{A}$ to $\mathbb{B}$ with various topologies and study convergence properties of sequences in $\mathcal{H}(\mathbb{A},\mathbb{B})$. We are in particular interested in the situation when $\mathbb{B}$ is a measure algebra as in this case we obtain a natural tool for studying topological convergence properties of sequences of ultrafilters on $\mathbb{A}$ in random extensions of the set-theoretical universe. This appears to have strong connections with Dow and Fremlin's result stating that there are Efimov spaces in the random model. We also investigate relations between topologies on $\mathcal{H}(\mathbb{A},\mathbb{B})$ for a Boolean algebra $\mathbb{B}$ carrying a strictly positive measure and convergence properties of sequences of measures on $\mathbb{A}$.

\end{abstract}

\section{Introduction}

It is a common issue in analysis and topology to compare convergence properties of sequences in topological spaces, in particular, to study when a sequence convergent with respect to one topology converges also with respect to another one. For
instance, one can ask what properties a Banach space must have so that its every weakly convergent sequence is also norm convergent (studying so called Schur property), or investigate when a given topological space does not contain non-trivial convergent sequences, i.e. when every convergent sequence is eventually constant. In this paper we also deal with this kind of problems---our setting are spaces of homomorphisms between Boolean algebras, especially measure algebras, endowed with several miscellaneous topologies. 

There are several motivations standing behind our research, but before we present them, together with our main results, we introduce some basic notations. 
For a cardinal $\kappa$ let $\mathbb{M}_\kappa$ denote the measure algebra of the standard product measure $\lambda_\kappa$ on the product space $2^\kappa$. If $\kappa=1$, then put $\mathbbm{2}=\mathbb{M}_1$. If $\mathbb{A}$ is a Boolean algebra, then by $\mathcal{H}(\mathbb{A},\mathbb{M}_\kappa)$ we denote the family of all Boolean homomorphisms from $\mathbb{A}$ into $\mathbb{M}_\kappa$. The measure $\lambda_\kappa$ induces the so-called Fr\'echet--Nikodym metric $d_{\lambda_\kappa}$ on $\mathbb{M}_\kappa$, and thus $\mathbb{M}_\kappa$ may be treated as a metric topological space. Since $\mathcal{H}(\mathbb{A},\mathbb{M}_\kappa)$ is simply a subset of the family of all functions from $\mathbb{A}$ into $\mathbb{M}_\kappa$, this further leads to endowing $\mathcal{H}(\mathbb{A},\mathbb{M}_\kappa)$ with two natural topologies---\textit{the pointwise topology} and \textit{the uniform topology}. 
In this article we will investigate these structures and their variations as well as the relations between them.

Our first motivation for this research is the observation that for $\kappa>1$ the space $\mathcal{H}(\mathbb{A},\mathbb{M}_\kappa)$ with the pointwise topology may be treated as a natural generalization of the Stone space of the Boolean algebra
$\mathbb{A}$, or a Stone space of $\mathbb{A}$ of \textit{higher} order. Indeed, recall that every ultrafilter on the Boolean algebra $\mathbb{A}$ induces a homomorphism from $\mathbb{A}$ into the two-element algebra $\mathbbm{2}$, and \textit{vice
versa}, and thus the Stone space $St(\mathbb{A})$ of $\mathbb{A}$ may be seen as the space $\mathcal{H}(\mathbb{A},\mathbbm{2})$, endowed with the pointwise topology originating from the Fr\'echet--Nikodym metric $d_{\lambda_1}$, induced by the
trivial measure $\lambda_1$ on $\mathbbm{2}$. As the literature concerning convergence properties of sequences in Stone spaces of Boolean algebras, or more generally in totally disconnected topological spaces, is quite vast, cf. e.g. \cite{BrianDow}, \cite{KPHart}, \cite{HruShi}, we were motivated to conduct similar research in the setting of the spaces $\mathcal{H}(\mathbb{A},\mathbb{M}_\kappa)$ for general $\kappa$'s and various topologies. Note that similar generalizations of topological spaces are quite common in analysis, e.g. one can think of Radon measures of norm $1$ on compact spaces as generalizations of points in those spaces, or, as it is done in non-commutative topology, of projections in general C*-algebras as being analogous to characteristic functions of clopen subsets of locally compact spaces.

For every $\kappa$ the uniform topology on the space $\mathcal{H}(\mathbb{A},\mathbb{M}_\kappa)$ is of course finer than the pointwise topology, hence every uniformly convergent sequence in $\mathcal{H}(\mathbb{A},\mathbb{M}_\kappa)$ is pointwise
metric convergent. As one can suspect, the converse in general does not hold---cf. Section \ref{sec-examples} for relevant counterexamples (especially interesting is the example described in Section \ref{ex_borel_not_unif_alg}, to obtain which we used tools from information theory such as perfect Hamming codes). In Section \ref{sec-algebraic_conv} we introduce and study a strengthening of pointwise metric
convergence---called by us \textit{the pointwise algebraic convergence}---that implies under certain conditions put on the Boolean algebra $\mathbb{A}$ (such as $\sigma$-completeness) also the uniform convergence (Corollary \ref{alg_impl_unif}). This
mode of convergence is related to the well-known notion of the algebraic convergence in $\sigma$-complete Boolean algebras, studied e.g. in \cite{Glowczynski}, \cite{BalFraHru}, \cite{JechSub}.

In Section \ref{sec-borel_conv}, using the dual functions to homomorphisms from $\mathbb{A}$ into $\mathbb{M}_\kappa$ and Borel subsets of the Stone space $\mathbb{A}$, we endow the space $\mathcal{H}(\mathbb{A},\mathbb{M}_\kappa)$ with yet another
topology, which we call \textit{the pointwise Borel metric topology}. This topology is finer than the topology of pointwise convergence and has turned out to be useful in investigations related to our second motivation which is the fact that each homomorphism $\varphi\in \mathcal{H}(\mathbb{A},\mathbb{M}_\kappa)$ naturally induces a measure on $\mathbb{A}$ of the form $\lambda_\kappa \circ \varphi$, as well as, by the virtue of the celebrated Maharam theorem, for each strictly positive measure $\mu$ on $\mathbb{A}$ there are a cardinal $\gamma$ and a (usually not unique) homomorphism $\psi\in\mathcal{H}(\mathbb{A},\mathbb{M}_\gamma)$  such that $\mu=\psi\circ\lambda_\gamma$. This correspondence has inspired us to study relations between sequences $(\varphi_n)$ of homomorphisms in $\mathcal{H}(\mathbb{A},\mathbb{M}_\kappa)$ and sequences of measures of the form $\lambda_\kappa\circ\varphi_n$. In particular, we show in Corollary \ref{convergences_measures_implications} connections between the uniform convergence, pointwise Borel metric convergence, and pointwise metric convergence of a sequence $(\varphi_n)$ in $\mathcal{H}(\mathbb{A},\mathbb{M}_\kappa)$ and the norm convergence, weak convergence, and weak* convergence of the sequence $(\lambda_\kappa\circ\varphi_n)$ of measures on $\mathbb{A}$, respectively. In Section \ref{measures} we prove also that whenever $\mathbb{A}$ has the Grothendieck property (see the end of Section \ref{sec-top_measures} for the definition), then every pointwise metric convergent sequence in $\mathcal{H}(\mathbb{A},\mathbb{M}_\kappa)$ is pointwise Borel metric convergent---a fact that does not hold for general Boolean algebras $\mathbb{A}$ (cf. the example in Section \ref{ex_alg_not_borel}).

Our last main motivation lies in the theory of forcing, where actually our interest in studying the subject of this article started. Imagine that we force with a Boolean algebra $\mathbb{B}$ over the ground model $V$ and that we are interested in the behaviour of sequences of ultrafilters in the forcing extension $V^\mathbb{B}$ which are defined on a fixed ground model Boolean algebra $\mathbb{A}$. This is actually a quite common situation---e.g. reals in the extension can be seen as ultrafilters on the Cantor algebra $\mathbb{C}$, i.e. the free countable Boolean algebra. 
The usual problem is that the access to objects in forcing extensions is quite \emph{remote}, through $\mathbb{B}$-names only. However, 
the  $\mathbb{B}$-names for ultrafilters on $\mathbb{A}$ naturally correspond to homomorphisms $\varphi\colon \mathbb{A} \to \mathbb{B}$ (see Section \ref{sec-homo-names}). 
 Then, the convergence of ultrafilters (treated as elements of the Stone space of $\mathbb{A}$ in $V^\mathbb{B}$) appears to be strongly connected to certain convergence properties of sequences of homomorphisms in the ground model. Namely, we prove
 in Proposition \ref{convergence} that for a given sequence $(\dot{\mathcal{U}}_n)$ of $\mathbb{M}_\kappa$-names for ultrafilters on $\mathbb{A}$ the sequence of the corresponding homomorphisms $(\varphi_{\dot{\mathcal{U}}_n})$ in
 $\mathcal{H}(\mathbb{A},\mathbb{M}_\kappa)$ is algebraically convergent if and only if the sequence $((\dot{\mathcal{U}}_n)_G)$ of the interpretations converges in $V[G]$ for every $\mathbb{M}_\kappa$-generic filter $G$. It turns out that this
 proposition has also variants regarding uniform convergence of sequences of homomorphisms---Proposition \ref{uniform-convergence} asserts that if $((\dot{\mathcal{U}}_n)_G)$ is eventually constant in every generic extension $V[G]$, then the sequence
 $(\varphi_{\dot{\mathcal{U}}_n})$ is uniformly convergent. The converse holds partially---in Theorem \ref{uniform-convergence2} we show that if a sequence of homomorphisms $(\varphi_n)$ in $\mathcal{H}(\mathbb{A},\mathbb{M}_\kappa)$ is uniformly
 convergent to some $\varphi\in\mathcal{H}(\mathbb{A},\mathbb{M}_\kappa)$, then for almost all $n\in\omega$ and the corresponding names $\dot{\mathcal{U}}_{\varphi_n}$ and $\dot{\mathcal{U}}_\varphi$ for ultrafilters on $\mathbb{A}$ there is a
 condition $p_n\in\mathbb{M}_\kappa$ forcing that $\dot{\mathcal{U}}_{\varphi_n}=\dot{\mathcal{U}}_\varphi$. This latter theorem is proved with the aid of Theorem \ref{distinct} asserting that if for two $\mathbb{M}_\kappa$-names $\dot{\mathcal{U}}$
 and $\dot{\mathcal{V}}$ for ultrafilters on $\mathbb{A}$ it holds $\Vdash_{\mathbb{M}_\kappa}\dot{\mathcal{U}}\neq\dot{\mathcal{V}}$, then there exists a \textit{large} condition $p\in\mathbb{M}_\kappa$, where \textit{large} means \textit{of measure
 as close to $1/4$ as one wishes}, and an element $A\in\mathbb{A}$ such that $p\Vdash A\in\dot{\mathcal{U}}\triangle\dot{\mathcal{V}}$---a result interesting on its own, since a typical argument based on the countable chain condition of
 $\mathbb{M}_\kappa$ provides us only with a countable antichain $\{p_n\colon\ n\in\omega\}$ of conditions witnessing that the two ultrafilters are different, but with no control over the value of $\lambda_\kappa(p_n)$ for any $n\in\omega$. The proof of Theorem \ref{distinct} is also interesting, because it boils down to purely combinatorial Proposition \ref{malowanie-plotu-new} having a nice real-life interpretation, see Remark \ref{rem-malowanie-plotu}.

The aforementioned results have connection with the famous long-standing open question, called \textit{the Efimov problem}, asking whether there exists \textit{an Efimov space}, that is, an infinite compact Hausdorff space which does not contain
copies of $\beta\omega$, the \v{C}ech--Stone compactification of $\omega$, nor non-trivial convergent sequences. No ZFC example of an Efimov space is known, however several consistent examples have been obtained either under some additional
set-theoretic assumptions, e.g. by Fedorchuk \cite{Fedorchuk76} (under Jensen's diamond principle), Dow and Pichardo-Mendoza \cite{DowPichardo} (under the Continuum Hypothesis), Dow and Shelah \cite{Dow-Shelah} (under Martin's axiom), or using forcing, see e.g. Sobota and Zdomskyy
\cite{Sobota-Zdomsky} or Dow and
Fremlin \cite{DowFremlin}. Results in the latter paper are of particular interest to us as they concern forcing with $\mathbb{M}_\kappa$; namely, Dow and Fremlin proved that if $\mathbb{A}$ is a ground model Boolean algebra such that its Stone space is an F-space, which
occurs e.g. in the case of $\sigma$-complete Boolean algebras or the algebra $\mathcal{P}(\omega)/Fin$, then, in $V^{\mathbb{M}_\kappa}$, the Stone space $St(\mathbb{A})$ does not have any non-trivial convergent sequences (which yields, e.g., that if
$V$ is a model of set theory satisfying the Continuum Hypothesis, then $St(\mathcal{P}(\omega)\cap V)$ is an Efimov space in any $\mathbb{M}_{\omega_2}$-generic extension of $V$). In Section \ref{secEfimov} we prove that $\mathbb{M}_\kappa$ forces
that the Stone space of a given ground model Boolean algebra $\mathbb{A}$ does not contain any non-trivial convergent sequences if and only if in $V$ every pointwise algebraically convergent sequence in $\mathcal{H}(\mathbb{A},\mathbb{M}_\kappa)$ is
uniformly convergent (Theorem \ref{translation}). This, together with Dow and Fremlin's theorem, implies that if the Stone space of an infinite Boolean algebra $\mathbb{A}$ is an F-space, then every pointwise algebraically convergent sequence in
$\mathcal{H}(\mathbb{A},\mathbb{M}_\kappa)$ is uniformly convergent---note that this result does not hold for 
'simple' Boolean algebras such as the Cantor algebra, see the example in Section \ref{ex_borel_alg_not_unif}.



\medskip

The structure of the paper is as follows. In the next section we present basic notations, terminology and facts used in the paper. In Section \ref{sec-convergence} we introduce four types of topologies and modes of convergence of sequences of homomorphisms, mostly into measure algebras. Section \ref{measures} is devoted to study relations between convergence properties of sequences $(\varphi_n)$ of homomorphisms into a Boolean algebra carrying a strictly positive measure $\mu$ and convergence properties of sequences of measures of the form $(\mu\circ\varphi_n)$. In Section \ref{sec-examples} we present a series of examples of sequences of homomorphisms between Boolean algebras being convergent with respect to one topology but not with respect to another one. In Section \ref{sec-homo-names} we study relations between various types of convergence of sequences of homomorphisms into measure algebras and convergence of corresponding ultrafilters in random generic extensions of the ground model. This study is continued in Section \ref{secEfimov}, where we characterize those Boolean algebras whose Stone spaces contain no non-trivial convergent sequences in random extensions with the aid of convergence properties of sequences of homomorphisms in the ground model. The last section provides several open questions and problems. 

\section{Acknowledgements}

We would like to thank Grzegorz Plebanek and Krzysztof Majcher for valuable discussions concerning homomorphisms, Hamming codes and painting fences. We also thank Lyubomyr Zdomskyy for reading the first draft of the paper and sharing with us many relevant comments.

\section{Notations and terminology}\label{preliminaries}

All \textit{compact} spaces considered in the paper are assumed to be Hausdorff. If $K$ is a compact space, then by $Clopen(K)$ and $Bor(K)$ we denote the Boolean algebra of clopen subsets of $K$ and the $\sigma$-field of all Borel subsets of $K$, respectively. If $\mathbb{A}$ is a Boolean algebra, then by $St(\mathbb{A})$ we denote its Stone space. Note that $\mathbb{A}$ and $Clopen(St(\mathbb{A}))$ are isomorphic.

Unless otherwise stated, all \textit{measures} considered by us are \textit{probability} measures. \textit{A measure on a Boolean algebra} is always meant to be finitely additive. On the other hand, \textit{a measure on a compact space} is always a Radon measure, i.e. it is countably additive, Borel and inner regular with respect to compact subsets. Note that every measure $\mu$ on a Boolean algebra $\mathbb{A}$ has a unique extension to a Radon measure $\widehat{\mu}$ on the Stone space $St(\mathbb{A})$, i.e. $\mu=\widehat{\mu}\restriction Clopen(St(\mathbb{A}))$, where $\mathbb{A}$ is identified with $Clopen(St(\mathbb{A}))$. When there should be no confusion, we will usully drop $\ \widehat{ }\ $ and write simply $\mu$ for $\widehat{\mu}$.

If $\mathbb{A}$ and $\mathbb{B}$ are two Boolean algebras, then $\mathcal{H}(\mathbb{A},\mathbb{B})$ denotes the family of all homomorphisms from $\mathbb{A}$ to $\mathbb{B}$. In Section \ref{sec-convergence} we will endow $\mathcal{H}(\mathbb{A},\mathbb{B})$ with various topologies and consider various types of convergences of sequences in $\mathcal{H}(\mathbb{A},\mathbb{B})$. If $\varphi\in\mathcal{H}(\mathbb{A},\mathbb{B})$, then the function $f_\varphi\colon St(\mathbb{B})\to St(\mathbb{A})$ defined as $f_\varphi(x)=\varphi^{-1}[x]$ is continuous (here $x$ is an ultrafilter on $\mathbb{B}$). On the other hand, if $f\colon St(\mathbb{B})\to St(\mathbb{A})$ is a continuous function, then the function $\varphi_f\colon\mathbb{A}\to\mathbb{B}$ defined by the formula $\varphi_f(A)=f^{-1}[A]$ is a homomorphism between Boolean algebras $\mathbb{A}$ and $\mathbb{B}$. It follows that $f=f_{(\varphi_f)}$ and $\varphi=\varphi_{(f_\varphi)}$.

\subsection{Topologies on spaces of measures.}\label{sec-top_measures} Let $\mathbb{A}$ be a Boolean algebra and let $K$ be a compact space. 
\textit{The space of all measures} on $K$ is denoted by $P(K)$. By $C(K)$ we denote the Banach space of continuous real-valued functions on $K$ endowed with the supremum topology. If $\mu\in P(K)$ and $f\in C(K)$, then $\mu(f)$ is defined as the integral $\mu(f)=\int_Kfd\mu$. 

 We endow $P(K)$ with three topologies: the norm topology, the weak topology and the weak* topology. \textit{The norm topology} on $P(K)$ is induced by the variation metric $d_{var}$ on $P(K)$ defined by the formula
\[d_{var}(\mu,\nu)=\sup_{\substack{A,B\in Bor(K)\\A\cap B=\emptyset}}\big(|\mu(A)-\nu(A)|+|\mu(B)-\nu(B)|\big)\]
for every $\mu,\nu\in P(K)$. Note that if $K$ is totally-disconnected, then in the above formula for $d_{var}(\mu,\nu)$ we may confine ourselves only to pairs of disjoint clopen subsets of $K$. 

\textit{The weak topology} on the dual space $C(K)^*$ (and hence on $P(K)$) is the weakest topology which makes all functionals from $C(K)^*$ continuous and so it is induced by the subbase given by sets of the form
\[V(\mu,\varphi,\eps)=\big\{\nu\in P(K)\colon\ |\varphi(\mu) - \varphi(\nu)|<\eps\big\},\]
where $\mu\in P(K)$, $\varphi\in C(K)^{**}$ and $\eps>0$. Similarly, \textit{the weak* topology} on $P(K)$ is defined by sets of the form
\[V(\mu,f,\eps)=\big\{\nu\in P(K)\colon\ |\mu(f)-\nu(f)|<\eps\big\},\]
where $\mu\in P(K)$, $f\in C(K)$ and $\eps>0$. Recall that $P(K)$ with the weak* topology is a compact space. Of course, the weak* topology on $P(K)$ is weaker than the weak topology which is on the other hand weaker than the norm topology.

The following facts are well-known.

\begin{prop}Let $K$ be a totally disconnected compact space. Let $\mathcal{V}$ be the collection of the sets of the form: 
	\[ V(\mu, C, \varepsilon) = \big\{\nu \in P(K)\colon |\nu(C) - \mu(C)|<\varepsilon\big\}, \]
		where $\mu \in P(K)$, $C$ is a clopen subset of $K$ and $\varepsilon>0$. Then $\mathcal{V}$ is a subbase of the weak* topology on $P(K)$.
\end{prop}
%

\begin{prop}\cite[Theorem 11 in Section VII]{Diestel}\label{diestel-weak} Let $K$ be a compact space. Then the sequence $(\mu_n)$ of measures on $K$ converges weakly to some $\mu$ if and only if $\mu_n(B)$ converges to $\mu(B)$ for each Borel subset $B\subseteq K$.
\end{prop}

We say that a Boolean algebra $\mathbb{A}$ has \textit{the Grothendieck property} if every weak* convergent sequence of (signed) Radon measures on $St(\mathbb{A})$ is weakly convergent. It is well-known that $\sigma$-complete Boolean algebras have the property, but countable ones do not (or more generally those algebras whose Stone spaces contain non-trivial convergent sequences).


\subsection{Measure algebras.} If $\mu$ is a measure on a Boolean algebra $\mathbb{A}$, then we put $\mathcal{N}_\mu=\big\{A\in\mathbb{A}\colon\ \mu(A)=0\big\}$. Similarly, if $\mu$ is a Radon measure on a compact space $K$, then we denote $\mathcal{N}_\mu=\big\{A\in Bor(K)\colon\ \mu(A)=0\big\}$. For an infinite cardinal number $\kappa$, by $\mathbb{M}_\kappa$ we will denote the (standard) measure algebra of Maharam type $\kappa$, i.e. $\mathbb{M}_\kappa=\mathrm{Bor}(2^\kappa)/\mathcal{N}_\kappa$, where $\mathcal{N}_\kappa$ is the $\sigma$-ideal of null sets with respect to the standard product measure $\lambda_\kappa$ on the space $2^\kappa$, i.e. $\mathcal{N}_\kappa=\mathcal{N}_{\lambda_\kappa}$. By $\lambda$ and $\mathbb{M}$ we mean simply $\lambda_\omega$ and $\mathbb{M}_\omega$, respectively.

By $\mathbbm{2}$ we mean the 2-point Boolean algebra $\{0,1\}$. Notice that $\mathbbm{2} = \mathbb{M}_1$, i.e. it is a measure algebra (with measure $\lambda_1$).

$\mathbb{C}$ will denote \textit{the Cantor algebra}, i.e. the free algebra generated by $\omega$ generators. Note that $\mathbb{C}$ is isomorphic to the algebra of clopen subsets of the Cantor space $2^\omega$ and hence $St(\mathbb{C})$ is homeomorphic to $2^\omega$. We will use this identification frequently.

In general we denote elements of Boolean algebras, including $\mathbb{M}_\kappa$'s, with capital letters. However, speaking about an element of $\mathbb{M}_\kappa$ sometimes we understand it as a condition of the forcing notion $\mathbb{M}_\kappa
\setminus \{0\}$. In this case, we will rather use the standard forcing notation: $p$, $q$ and so on.

\subsection{Metric Boolean algebras.} 
We will say, after Kolmogorov \cite{Kolmogorov}, that a Boolean algebra $\mathbb{B}$ is \emph{metric} if
it supports a strictly positive measure, i.e. there is a  measure $\mu$ on $\mathbb{B}$ such that $\mu(A)>0$ for every $A\in\mathbb{B}\setminus\{0\}$. Note that a metric algebra need not to be $\sigma$-complete, so it is not necessarily a measure algebra. The word \emph{metric} is explained by the following simple fact.

\begin{fact} If $\mu$ is a strictly positive measure on a Boolean algebra $\mathbb{B}$, then the function $d_\mu\colon \mathbb{B}\times\mathbb{B} \to \mathbb{R}$ defined for every $A,B\in\mathbb{B}$ by the formula
	\[ d_\mu(A,B) = \mu(A\triangle B) \]
is a metric on $\mathbb{B}$ (called \emph{Fr\'echet--Nikodym metric}).
\end{fact}


Note that we may define easily a Radon version of the Fr\'echet--Nikodym metric. Namely, if $\mu$ is a strictly positive measure on a Boolean algebra $\mathbb{B}$, then the function $d_\mu^{Bor}\colon Bor(St(\mathbb{B}))\times Bor(St(\mathbb{B}))\to\mathbb{R}$ defined as
\[d_\mu^{Bor}(A,B) = \widehat{\mu}(A\triangle B),\]
where $A,B\in Bor(St(\mathbb{B}))$, is a pseudometric. Call $d_\mu^{Bor}$ \textit{the Borel Fr\'echet--Nikodym pseudometric} on $(\mathbb{B},\mu)$. 

Recall that there is a natural and well-studied notion of convergence in complete Boolean algebras.  Let $\mathbb{A}$ be a complete Boolean algebra and let $(A_n)$ be a sequence of its elements. We say that $(A_n)$ \emph{algebraically converges} to $A\in \mathbb{A}$ if 
\[\bigvee_n \bigwedge_{m>n} A_m = \bigwedge_n \bigvee_{m>n} A_m = A.\]
Such a notion was deeply studied e.g. in \cite{Glowczynski}, \cite{BalFraHru}, \cite{JechSub}. \emph{The sequential topology} on $\mathbb{A}$ is the largest topology with respect to which all the algebraically convergent sequences converge. If $\mathbb{A}$ is a measure algebra, then the sequential topology coincides with the topology introduced by the Fr\'echet--Nikodym metric (see \cite{Glowczynski}).

\section{Convergences and topologies in $\mathcal{H}(\mathbb{A},\mathbb{B})$}\label{sec-convergence}

In the next subsections we will see that for given Boolean algebras $\mathbb{A}$ and $\mathbb{B}$ we may endow the space $\mathcal{H}(\mathbb{A},\mathbb{B})$ with several topologies and convergences.

\subsection{Pointwise metric topology}

Let $(\mathbb{B},\mu)$ be a metric algebra. Then, there is a natural topology on $\mathcal{H}(\mathbb{A},\mathbb{B})$ given by the metric $d_\mu$, which we call \textit{the pointwise metric topology}. Namely, an element of the subbase is of the form
\[ V(\varphi,A,\varepsilon) = \big\{\psi\in \mathcal{H}(\mathbb{A},\mathbb{B})\colon d_\mu(\varphi(A),\psi(A))<\varepsilon\big\}\]
for $\varphi\in\mathcal{H}(\mathbb{A},\mathbb{B})$, $A\in \mathbb{A}$ and $\varepsilon>0$.

Let $(\varphi_n)$ be a sequence of homomorphisms from $\mathbb{A}$ to $\mathbb{B}$. It follows that $(\varphi_n)$ is convergent in the pointwise metric topology to some $\varphi\in\mathcal{H}(\mathbb{A},\mathbb{B})$ if and only if for every $A\in\mathbb{A}$ and every $\eps>0$ there is $N\in\omega$ such that $d_\mu\big(\varphi_n(A),\varphi(A)\big)<\eps$ for every $n>N$---in this case, we will say that $(\varphi_n)$ is \textit{pointwise metric convergent} to $\varphi$. Note that if $(\varphi_n)$ is not convergent to $\varphi$, then there are $A\in\mathbb{A}$, $\eps>0$ and a subsequence $(\varphi_{n_k})$ such that for every $k\in\omega$ we have
\[\eps\le d_\mu\big(\varphi_{n_k}(A),\varphi(A)\big)=\mu\big(\varphi_{n_k}(A)\setminus\varphi(A)\big)+\mu\big(\varphi(A)\setminus\varphi_{n_k}(A)\big),\]
so either $\mu\big(\varphi_{n_k}(A)\setminus\varphi(A)\big)\ge\eps/2$ for infinitely many $k\in\omega$, or $\mu\big(\varphi(A)\setminus\varphi_{n_k}(A)\big)\ge\eps/2$ for infinitely many $k\in\omega$. Since
\[\mu\big(\varphi(A)\setminus\varphi_{n_k}(A)\big)=\mu\big(\varphi_{n_k}(A^c)\setminus\varphi(A^c)\big),\]
it follows that if $(\varphi_n)$ does not converge pointwise metric to $\varphi$, then we may always find $A\in\mathbb{A}$, $\eps>0$ and a subsequence $(\varphi_{n_k})$ such that $\mu\big(\varphi_{n_k}(A)\setminus\varphi(A)\big)\ge\eps$. We will use this observation regularly.

\begin{rem}\label{hom_2_stone}
Using the natural bijection between the set of ultrafilters on a Boolean algebra $\mathbb{A}$ and the set of homomorphisms from
$\mathcal{H}(\mathbb{A},\mathbbm{2})$, we see that the pointwise metric topology coincides with the Stone topology on the set of	ultrafilters on $\mathbb{A}$, i.e. $St(\mathbb{A})$ and $\mathcal{H}(\mathbb{A},\mathbbm{2})$ are homeomorphic. 
Thus, since $\mathbbm{2}$ embeds into every (metric) Boolean algebra $\mathbb{B}$, the space $\mathcal{H}(\mathbb{A},\mathbb{B})$ with the pointwise metric topology always contains a closed copy of the Stone space of $\mathbb{A}$, which is immediately yielded by the following trivial fact. 
\end{rem}

\begin{lem}
Let $\mathbb{A}$, $\mathbb{B}$ and $\mathbb{C}$ be Boolean algebras. Assume that $(\mathbb{B},\mu)$ and $(\mathbb{C},\nu)$ are metric and such that $\mathbb{B}$ is a subalgebra of $\mathbb{C}$ with $\mu=\nu\restriction\mathbb{B}$. Endow $\mathcal{H}(\mathbb{A},\mathbb{B})$ and $\mathcal{H}(\mathbb{A},\mathbb{C})$ with the pointwise metric topologies. If for every $A\in\mathbb{C}\setminus\mathbb{B}$ there is $\eps>0$ such that $d_\nu(A,B)<\eps$ for no $B\in\mathbb{B}$, then $\mathcal{H}(\mathbb{A},\mathbb{B})$ is closed in $\mathcal{H}(\mathbb{A},\mathbb{C})$.
\end{lem}

\subsection{Uniform topology}

Let us again assume that $(\mathbb{B},\mu)$ is a metric algebra. We may define a metric $d_{hom}$ on the set $\mathcal{H}(\mathbb{A},\mathbb{B})$  in the following way:
\[d_{hom}(\varphi,\psi) = \sup\big\{d_\mu(\varphi(A),\psi(A))\colon A\in \mathbb{A}\big\},\]
where $\varphi,\psi\in\mathcal{H}(\mathbb{A},\mathbb{B})$. According to the topology induced by $d_{hom}$, which we call \textit{the uniform topology}, a sequence $(\varphi_n)$ converges to $\varphi$ iff it converges uniformly with respect to the metric $d_\mu$, i.e. for each $\varepsilon>0$ there is $N$ such that for every $n>N$ we have $d_\mu\big(\varphi_n(A),\varphi(A)\big)<\varepsilon$ for each $A\in \mathbb{A}$. We will say in this case that $(\varphi_n)$ \textit{converges uniformly} to $\varphi$. Of course, if a sequence of homomorphisms converges uniformly, then it converges pointwise---the converse however does not necessarily hold, see Sections \ref{ex_alg_not_borel}, \ref{ex_metric_not_borel_alg} and \ref{ex_borel_not_unif_alg} for appropriate examples.

\begin{rem}
Contrary to the pointwise metric topology (see Remark \ref{hom_2_stone}), for every Boolean algebra $\mathbb{A}$ the uniform topology on $\mathcal{H}(\mathbb{A},\mathbbm{2})$ is discrete.
\end{rem}

\subsection{Pointwise algebraic convergence}\label{sec-algebraic_conv}

This kind of convergence can be introduced in $\mathcal{H}(\mathbb{A},\mathbb{B})$ if $\mathbb{B}$ is $\sigma$-complete (but not necessarily metric). We say that a sequence of homomorphisms $(\varphi_n)$ from $\mathbb{A}$ to $\mathbb{B}$ converges \textit{pointwise algebraically} to $\varphi$ if for every $A\in \mathbb{A}$ we have
\[\bigvee_n \bigwedge_{m>n} \varphi_m(A) = \bigwedge_n \bigvee_{m>n} \varphi_m(A) = \varphi(A).\] 
In fact, by the de Morgan laws,  $(\varphi_n)$ converges pointwise algebraically to $\varphi$ if and only if any of the following two equivalent conditions hold:
\begin{itemize}
	\item for every $A\in \mathbb{A}$ we have $\bigvee_n \bigwedge_{m>n} \varphi_m(A) = \varphi(A)$;
	\item for every $A\in \mathbb{A}$ we have $\bigwedge_{n}\bigvee_{m>n}  \varphi_m(A) = \varphi(A)$.
\end{itemize}
If we assume that $\mathbb{B}$ is metric (with a strictly positive measure $\mu$), then if
$(\varphi_n)$ converges pointwise algebraically to $\varphi$, then for each $A\in \mathbb{A}$ we have
\[\lim_{n\to\infty} \mu\big(\bigwedge_{m>n} \varphi_m(A)\big) = \lim_{n\to\infty} \mu\big(\bigvee_{m>n} \varphi_m(A)\big) =  \mu(\varphi(A)).\]

The following fact binds the pointwise algebraic convergence with the pointwise metric convergence. Recall that for every $\kappa$ the measure $\lambda_\kappa$ is \textit{continuous}, i.e. for every decreasing sequence $(A_n)$ in $\mathbb{M}_\kappa$ we have $\lim_{n\to\infty}\lambda_\kappa(A_n)=\lambda_\kappa\big(\bigwedge_{n\in\omega}A_n\big)$.

\begin{prop}\label{algebraic_metric}
	Let $\mathbb{A}$ be a Boolean algebra and $\kappa$ a cardinal number. Let $(\varphi_n)$ be a sequence of homomorphisms from $\mathbb{A}$ to $\mathbb{M}_\kappa$. If $(\varphi_n)$ is pointwise algebrically convergent to some $\varphi\in\mathcal{H}(\mathbb{A},\mathbb{M}_\kappa)$, then $\varphi_n$ converges pointwise metric to $\varphi$.
\end{prop}
\begin{proof}
Let $A\in\mathbb{A}$. We will show that $\lim_{n\to\infty}\lambda_\kappa\big(\varphi_n(A)\triangle\varphi(A)\big)=0$. For the sake of contradiction assume that there is $\eps>0$ and a subsequence $(\varphi_{n_k})$ such that
$\lambda_\kappa(\varphi_{n_k}(A)\setminus\varphi(A))\ge\eps$ for every $k\in\omega$. Then, by the continuity of $\lambda_\kappa$,
\[\lambda_\kappa\Big(\bigwedge_{k\in\omega}\bigvee_{l>k}\big(\varphi_{n_l}(A)\setminus\varphi(A)\big)\Big) \geq \varepsilon. \]
However, since $\bigwedge_{k\in\omega}\bigvee_{l>k}\varphi_{n_l}(A)=\varphi(A)$, we get that $\bigwedge_{k\in\omega}\bigvee_{l>k}\big(\varphi_{n_l}(A)\setminus\varphi(A)\big)=0$, so
\[0=\lambda_\kappa\Big(\bigwedge_{k\in\omega}\bigvee_{l>k}\big(\varphi_{n_l}(A)\setminus\varphi(A)\big)\Big)\ge\eps>0,\]
a contradiction.
\end{proof}

As we will see in Sections \ref{ex_unif_not_alg} and \ref{ex_metric_not_borel_alg}, the converse to Proposition \ref{algebraic_metric} does not hold in general. However, for $\mathbb{B}=\mathbbm{2}$ we have the following corollary. 

\begin{cor}\label{alg_pointwise_coincide}
Let $\mathbb{A}$ be a Boolean algebra. Then, in $\mathcal{H}(\mathbb{A},\mathbbm{2})$ the pointwise algebraic
	convergence coincides with the pointwise metric convergence.
\end{cor}
\begin{proof}
Let $(\varphi_n)$ be a sequence in $\mathcal{H}(\mathbb{A},\mathbbm{2})$. If $(\varphi_n)$ is pointwise algebraically convergent, then by Proposition \ref{algebraic_metric} it is also pointwise metric convergence.

Conversely, if $(\varphi_n)$ is pointwise metric convergent to $\varphi\in\mathcal{H}(\mathbb{A},\mathbbm{2})$, then for every $A\in\mathbb{A}$ there is $N\in\omega$ such that $\varphi_n(A)=\varphi(A)$ for every $n>N$. It follows that $\bigwedge_{n>N}\varphi_n(A)=\varphi(A)$, so in particular $\bigvee_{N\in\omega}\bigwedge_{n>N}\varphi_n(A)=\varphi(A)$. Hence, $(\varphi_n)$ converges pointwise algebraically to $\varphi$.
\end{proof}

\subsection{Pointwise Borel metric topology and uniform Borel topology}\label{sec-borel_conv}

Let $\mathbb{A}$ be a Boolean algebra and $(\mathbb{B},\mu)$ a metric Boolean algebra. To introduce yet another topology on $\mathcal{H}(\mathbb{A},\mathbb{B})$, called by us \textit{pointwise Borel metric topology}, we need to appeal to dual continuous functions and the pseudometric $d_\mu^{Bor}$. Namely, we define the subbase of the topology to be given by the sets of the form
\[V^{Bor}(\varphi,A,\eps)=\big\{\psi\in \mathcal{H}(\mathbb{A},\mathbb{B})\colon d_\mu^{Bor}\big(f_\varphi^{-1}[A],f_\psi^{-1}[A]\big)<\varepsilon\big\}\]
for $\varphi\in\mathcal{H}(\mathbb{A},\mathbb{B})$, $A\in Bor(St(\mathbb{A}))$ and $\varepsilon>0$.

Let $(\varphi_n)$ be a sequence of homomorphisms from $\mathbb{A}$ to $\mathbb{B}$. For each $n\in\omega$ let $f_n=f_{\varphi_n}$ and let $f=f_\varphi$. We say that $(\varphi_n)$ \textit{converges pointwise Borel metric} to $\varphi$ if $\lim_{n\to\infty}\widehat{\mu}\big(f_n^{-1}[B]\triangle f^{-1}[B]\big)=0$ for every $B\in Bor(St(\mathbb{A}))$. Obviously, if a sequence of homomorphisms converges pointwise Borel metric, then it converges pointwise metric, however the converse may not hold---we provide relevant counterexamples in Sections \ref{ex_alg_not_borel} and \ref{ex_metric_not_borel_alg} as well as some positive results in Section \ref{measures} (see, in particular, Corollary \ref{positive_grothendieck}). 

Similarly to the uniform topology on $\mathcal{H}(\mathbb{A},\mathbb{B})$, we may define \textit{the uniform Borel topology}. Namely, we define a Borel version $d_{hom}^{Bor}$ of the metric $d_{hom}$:
\[d_{hom}^{Bor}(\varphi,\psi) = \sup\big\{d_\mu^{Bor}\big(f_\varphi^{-1}[A],f_\psi^{-1}[A]\big)\colon A\in Bor(St(\mathbb{A}))\big\}, \]
where $\varphi,\psi\in\mathcal{H}(\mathbb{A},\mathbb{B})$. Note that, since $\mu$ is strictly positive on $\mathbb{A}$, $d_{hom}^{Bor}$ is a metric, even though $d_\mu^{Bor}$ is only a pseudometric. As before, we will say that a sequence $(\varphi_n)$ \textit{converges uniformly Borel} to $\varphi$ if for every $\eps>0$ there is $N\in\omega$ such that $d_{hom}^{Bor}(\varphi_n,\varphi)<\eps$ for every $n>N$, i.e. $\widehat{\mu}\big(f_n^{-1}[B]\triangle f^{-1}[B]\big)<\eps$ for every $n>N$ and every Borel subset $B\subseteq St(\mathbb{A})$. Of course, if a sequence converges uniformly Borel, then it converges pointwise Borel metric and uniformly. It appears that, conversely, the uniform convergence easily implies the Borel uniform convergence, because, in fact, the uniform topology and uniform Borel topology coincide.

\begin{lem}\label{uniform_metrics_are_equal}
Let $\mathbb{A}$ be a Boolean algebra and $(\mathbb{B},\mu)$ a metric Boolan algebra. Then, for the metrics $d_{hom}$ and $d_{hom}^{Bor}$ on $\mathcal{H}(\mathbb{A},\mathbb{B})$ we have $d_{hom}=d_{hom}^{Bor}$.
\end{lem}
\begin{proof}
Fix $\varphi,\psi\in\mathcal{H}(\mathbb{A},\mathbb{B})$. Since every clopen set in $St(\mathbb{A})$ is Borel, we have $d_{hom}(\varphi,\psi)\le d_{hom}^{Bor}(\varphi,\psi)$. Let $\eps>0$---we will show that $d_{hom}^{Bor}(\varphi,\psi)\le d_{hom}(\varphi,\psi)+\eps$, which will prove that $d_{hom}^{Bor}(\varphi,\psi)\le d_{hom}(\varphi,\psi)$. By the regularity of $\widehat{\mu}\circ f_\varphi^{-1}$ and $\widehat{\mu}\circ f_\psi^{-1}$, for every $B\in Bor(St(\mathbb{A}))$ there is $A_B\in\mathbb{A}$ such that $B\subseteq A_B$, $\widehat{\mu}\big(f_\varphi^{-1}[A_B\setminus B]\big)<\eps/2$ and $\widehat{\mu}\big(f_\psi^{-1}[A_B\setminus B]\big)<\eps/2$. Thus, for every Borel $B\in Bor(St(\mathbb{A}))$ we have:
\[\widehat{\mu}\big(f_\varphi^{-1}[B]\triangle f_\psi^{-1}[B]\big)\le\]
\[\widehat{\mu}\big(f_\varphi^{-1}[B]\triangle f_\varphi^{-1}\big[A_B\big]\big)+\widehat{\mu}\big(f_\varphi^{-1}\big[A_B\big]\triangle f_\psi^{-1}\big[A_B\big]\big)+\widehat{\mu}\big(f_\psi^{-1}\big[A_B\big]\triangle f_\psi^{-1}[B]\big)<\]
\[\eps/2+d_\mu\big(\varphi\big(A_B\big),\psi\big(A_B\big)\big)+\eps/2\le d_{hom}(\varphi,\psi)+\eps,\]
so $d_{hom}^{Bor}(\varphi,\psi)\le d_{hom}(\varphi,\psi)+\eps$.
\end{proof}

\begin{cor}\label{uniform_metrics_convergence}
Let $\mathbb{A}$ be a Boolean algebra and $(\mathbb{B},\mu)$ a metric Boolan algebra. Then:
\begin{enumerate}
	\item the uniform topology and the uniform Borel topology coincide;
	\item a sequence of homomorphisms from $\mathbb{A}$ to $\mathbb{B}$ converges uniformly if and only if it converges uniformly Borel.
\end{enumerate}
\end{cor}

\section{Homomorphisms and measures}\label{measures}


By the Maharam theorem for every measure $\mu$ on a Boolean algebra $\mathbb{A}$ there are a cardinal number $\kappa$ and an injective homomorphism $\varphi\colon\mathbb{A}\to\mathbb{M}_\kappa$ such that $\mu = \lambda_\kappa \circ   \varphi $. On the other hand, if $\varphi\colon \mathbb{A} \to \mathbb{M}_\kappa$ is a homomorphism, then it is plain to check that the function $\lambda_\kappa \circ \varphi \colon \mathbb{A} \to \mathbb{R}$ is a (finitely additive) measure. It follows that there is a natural mapping from the
family of homomorphisms $\mathcal{H}(\mathbb{A},\mathbb{M}_\kappa)$ into the family of measures on $\mathbb{A}$. 
This observation allows us to use notions from the theory of measures on Boolean algebras to study $\mathcal{H}(\mathbb{A},\mathbb{M}_\kappa)$.

\begin{prop}\label{weakstar_cont}
Let $\mathbb{A}$ be a Boolean algebra and $(\mathbb{B},\mu)$ a metric algebra. Consider the space $\mathcal{H}(\mathbb{A},\mathbb{B})$ with the topology $\tau$ and the space $P(St(\mathbb{A}))$ with topology $\tau'$, both given below. Then, the function $F_\mu\colon \mathcal{H}(\mathbb{A},\mathbb{B}) \to P(St(\mathbb{A}))$, given by $F_\mu(\varphi) =\widehat{\mu}\circ f_\varphi^{-1}$, is continuous in each of the following cases: 
	\begin{enumerate}
		\item $\tau$ is the uniform topology and $\tau'$ is the norm topology;
		\item $\tau$ is the pointwise metric topology and $\tau'$ is the weak* topology.
	\end{enumerate}
\end{prop}
\begin{proof} 
We will use below the fact that $|\mu(C) - \mu(D)|\leq\mu(C\triangle D)$ for every 	$C,D\in \mathbb{B}$.

We first prove (1). Using the above inequality, for every $A\in\mathbb{A}$ we have
\[\big|\mu(\varphi(A)) - \mu(\psi(A))\big| \leq \mu\big(\varphi(A) \triangle \psi(A)\big),\] 
and thus, by the definition of the variation metric on $P(St(\mathbb{A}))$, it holds:
\[d_{var}\big(F_\mu(\varphi),F_\mu(\psi)\big)=\sup_{\substack{A,B\in\mathbb{A}\\A\wedge B=0_\mathbb{A}}}\Big(\big|F_\mu(\varphi)(A)-F_\mu(\psi)(A)\big|+\big|F_\mu(\varphi)(B)-F_\mu(\psi)(B)\big|\Big)\leq\]
\[2\sup_{A\in\mathbb{A}}\big|F_\mu(\varphi)(A)-F_\mu(\psi)(A)\big|\leq 2\sup_{A\in \mathbb{A}} d_\mu(\varphi(A),\psi(A)) = 2d_{hom}(\varphi,\psi), \]
which actually shows  that $F_\mu$ is $2$-Lipschitz.


To prove (2) let $\varphi\in\mathcal{H}(\mathbb{A},\mathbb{B})$ and fix an element $V$ of the subbase of the weak* topology on $P(St(\mathbb{A}))$ of the form $V=V(F_\mu(\varphi),A,\eps) = \big\{\nu\colon\big|F_\mu(\varphi)(A)-\nu(A)\big|<\varepsilon\big\}$, where $A\in \mathbb{A}$ and $\varepsilon>0$. 
	Put:
	\[U=\big\{\psi\in\mathcal{H}(\mathbb{A},\mathbb{B})\colon d_\mu(\varphi(A),\psi(A))<\eps\big\}.\]
	Then, $U$ is an open set in $\mathcal{H}(\mathbb{A},\mathbb{B})$. For every $\psi\in U$ we have:
	\[\big|F_\mu(\varphi)(A)-F_\mu(\psi)(A)\big|\leq d_\mu\big(\varphi(A),\psi(A)\big)<\eps,\]
	so $\varphi\in F_\mu[U]\subseteq V$, which proves that $F_\mu$ is continuous.

\end{proof}

Consequently, the equivalence relation on $\mathcal{H}(\mathbb{A},\mathbb{B})$ defined by the formula
\[\varphi \sim \psi \iff F_\mu(\varphi) = F_\mu(\psi),\]
has closed equivalence classes, provided that $\mathcal{H}(\mathbb{A},\mathbb{B})$ is endowed with any of the topologies $\tau$ mentioned in the proposition. 

\begin{cor}\label{convergences_measures_implications}
Let $\mathbb{A}$ be a Boolean algebra and $(\mathbb{B},\mu)$ a metric algebra. Let $\varphi_n\in\mathcal{H}(\mathbb{A},\mathbb{B})$, $n\in\omega$, and $\varphi\in\mathcal{H}(\mathbb{A},\mathbb{B})$. The following hold:
\begin{enumerate}
	\item if $(\varphi_n)$ converges uniformly to $\varphi$, then $\big(\widehat{\mu}\circ f_{\varphi_n}^{-1}\big)$ converges to $\widehat{\mu}\circ f_{\varphi}^{-1}$ in norm;

	\item if $(\varphi_n)$ converges pointwise Borel metric to $\varphi$, then $\big(\widehat{\mu}\circ f_{\varphi_n}^{-1}\big)$ converges weakly to $\widehat{\mu}\circ f_{\varphi}^{-1}$;
	\item if $(\varphi_n)$ converges pointwise metric to $\varphi$, then $\big(\widehat{\mu}\circ f_{\varphi_n}^{-1}\big)$ converges weakly* to $\widehat{\mu}\circ f_{\varphi}^{-1}$.
\end{enumerate}
\end{cor}

\begin{proof}
	(1) and (3) are direct consequences of Proposition \ref{weakstar_cont}. (2) follows from Proposition \ref{diestel-weak}.
\end{proof}

\begin{rem}
Note here that the converse to Corollary \ref{convergences_measures_implications} may not hold even in the simplest case when all homomorphisms $\varphi_n$'s are the same. Indeed, let $(\mathbb{B},\mu)$ be a metric algebra and let each $\varphi_n$ be the identity homomorphism on $\mathbb{B}$. If $\varphi\in\mathcal{H}(\mathbb{B},\mathbb{B})$ is a homomorphism such that $\mu(\varphi(A))=\mu(A)$ and $\varphi(A)\neq A$ for some $A\in\mathbb{B}$, then the sequence $(\varphi_n)$ does not converge to $\varphi$ in any of the considered topologies, even though $\mu(\varphi_n(B))=\mu(\varphi(B))$ for every $n\in\omega$ and $B\in\mathbb{B}$.
\end{rem}

We now prove that if a Boolean algebra $\mathbb{A}$ has the Grothendieck property, then every pointwise metric convergent sequence of homomorphisms from $\mathbb{A}$ into a metric algebra is also pointwise Borel metric convergent. Recall that a sequence $(\mu_k)$ of Radon measures on a compact Hausdorff space $K$ is \textit{uniformly countably additive} (also, \textit{uniformly exhaustive}) if for every descending sequence $(E_n)$ of Borel subsets of $K$ such that $\bigcap_nE_n=\emptyset$ and every $\eps>0$ there is $N\in\omega$ such that $\big|\mu_k(E_n)\big|<\eps$ for every $n\ge N$ and $k\in\omega$. Equivalently, there is $N\in\omega$ such that $\big|\mu_k(E_m)-\mu_k(E_n)\big|<\eps$ for every $n,m\ge N$ and $k\in\omega$ (see \cite[Chapter VII, Theorem 10]{Diestel} for other equivalent definitions). We also say that $(\mu_n)$ is \textit{weakly convergent} to a Radon measure $\mu$ on $K$ if $\lim_{n\to\infty}\mu_n(B)=\mu(B)$ for every Borel subset $B$ of $K$. The Nikodym Convergence Theorem (\cite[Chapter VII, page 90]{Diestel}) asserts that every weakly convergent sequence of Radon measures is uniformly countably additive.

\begin{prop}\label{unif_ctbl_add_point_borel}
Let $\mathbb{A}$ be a Boolean algebra and $(\mathbb{B},\mu)$ be a metric Boolean algebra. Let $(\varphi_n)$ be a sequence in $\mathcal{H}(\mathbb{A},\mathbb{B})$ pointwise metric convergent to a homomorphism $\varphi\in\mathcal{H}(\mathbb{A},\mathbb{B})$. If the sequence $\big(\widehat{\mu}\circ f_{\varphi_n}^{-1}\big)$ is uniformly countably additive, then $(\varphi_n)$ converges pointwise Borel metric to $\varphi$. 
\end{prop}
\begin{proof}
Write simply $f_n$ for $f_{\varphi_n}$ and $f$ for $f_\varphi$. Fix $\eps>0$ and a Borel subset $B$ of $St(\mathbb{A})$. Let $\delta=\eps/2$. We will show first that $\widehat{\mu}\big(f^{-1}[B]\setminus f_k^{-1}[B]\big)<\delta$ for almost all $k\in\omega$.

Since each measure $\widehat{\mu}\circ f_k^{-1}$ is (outer) regular on $St(\mathbb{A})$, for each $k\in\omega$ we can find  $A_k\in\mathbb{A}$ such that $B\subseteq A_k$ and
\[\tag{1}\widehat{\mu}\big(f_k^{-1}[A_k\setminus B]\big)<\delta/4.\]
By taking intersections, we may assume that $A_{k+1}\le A_k$ for every $k\in\omega$. By the uniform countable additiveness of $\big(\widehat{\mu}\circ f_n^{-1}\big)$, there is $k_0\in\omega$ such that 
\[\tag{2}\widehat{\mu}\big(f_k^{-1}[A_l\setminus A_k]\big)<\delta/4\]
for every $k>l>k_0$. Since obviously the measure $\widehat{\mu}\circ f^{-1}$ is also countably additive, there is $k_1\ge k_0$ such that
\[\tag{3}\widehat{\mu}\big(f^{-1}[A_l\setminus A_k]\big)<\delta/4\]
for every $k>l>k_1$. Now, fix $l=k_1+1$. By the pointwise metric convergence of $(\varphi_k)$ to $\varphi$, there is $k_2\ge l$ such that
\[\tag{4}\widehat{\mu}\big(f^{-1}[A_l]\triangle f_k^{-1}[A_l]\big)<\delta/4\]
for every $k>k_2$. Using (3), (4) and (2), for every $k>k_2$ we have:
\[\tag{5}\widehat{\mu}\big(f^{-1}[A_k]\triangle f_k^{-1}[A_k]\big)\le\]
\[\widehat{\mu}\big(f^{-1}[A_k]\triangle f^{-1}[A_l]\big)+\widehat{\mu}\big(f^{-1}[A_l]\triangle f_k^{-1}[A_l]\big)+\widehat{\mu}\big(f_k^{-1}[A_l]\triangle f_k^{-1}[A_k]\big)<\frac{3\delta}{4}.\]
(5) and (1) yield that for every $k>k_2$ it holds:
\[\widehat{\mu}\big(f^{-1}[B]\setminus f_k^{-1}[B]\big)\le\widehat{\mu}\big(f^{-1}[A_k]\setminus f_k^{-1}[B]\big)\le\widehat{\mu}\big(f^{-1}[A_k]\triangle f_k^{-1}[B]\big)\le\]
\[\widehat{\mu}\big(f^{-1}[A_k]\triangle f_k^{-1}[A_k]\big)+\widehat{\mu}\big(f_k^{-1}[A_k]\triangle f_k^{-1}[B]\big)\le\frac{3\delta}{4}+\frac{\delta}{4}=\delta,\]
so
\[\tag{6}\widehat{\mu}\big(f^{-1}[B]\setminus f_k^{-1}[B]\big)<\eps/2.\]
Changing $B$ to $B^c$ and doing exactly the same computations that led us to (6), we get that there is $k_3\ge k_2$ such that
\[\tag{7}\widehat{\mu}\big(f_k^{-1}[B]\setminus f^{-1}[B]\big)=\widehat{\mu}\big(f^{-1}[B^c]\setminus f_k^{-1}[B^c]\big)<\eps/2\]
for every $k>k_3$. Thus, by (6) and (7), for every $k>k_3$ we get:
\[\widehat{\mu}\big(f^{-1}[B]\triangle f_k^{-1}[B]\big)<\eps.\]
This proves that $(\varphi_n)$ converges pointwise Borel metric to $\varphi$.
\end{proof}


\begin{cor}\label{grothendieck}
Let $\mathbb{A}$ be a Boolean algebra and $(\mathbb{B},\mu)$ be a metric Boolean algebra. Let $(\varphi_n)$ be a sequence in $\mathcal{H}(\mathbb{A},\mathbb{B})$ pointwise metric convergent to a homomorphism $\varphi\in\mathcal{H}(\mathbb{A},\mathbb{B})$. If $\mathbb{A}$ has the Grothendieck property, then $(\varphi_n)$ converges pointwise Borel metric to $\varphi$.
\end{cor}
\begin{proof}
Assume that $\mathbb{A}$ has the Grothendieck property. Let $(f_k)$ and $f$ be the continuous functions dual to $(\varphi_k)$ and $\varphi$, respectively. 
%
%
Since $\lim_{k\to\infty}\widehat{\mu}\big(f_k^{-1}[A]\big)=\widehat{\mu}\big(f^{-1}[A]\big)$ for every $A\in\mathbb{A}$, the sequence $\big(\widehat{\mu}\circ f_k^{-1}\big)$ is weakly* convergent to $\widehat{\mu}\circ f^{-1}$ (as it is uniformly bounded). The Grothendieck property of $\mathbb{A}$ implies that $\big(\widehat{\mu}\circ f_k^{-1}\big)$ is weakly convergent to $\widehat{\mu}\circ f^{-1}$. By the Nikodym Convergence Theorem, the sequence is uniformly countably additive, so the conclusion follows by Proposition \ref{unif_ctbl_add_point_borel}.
\end{proof}

Notice here that the measures $\widehat{\mu}\circ f_k^{-1}$ considered in the proof of Corollary \ref{grothendieck} are \textit{non-negative}, i.e. $\widehat{\mu}\circ f_k^{-1}(B)\ge0$ for every $B\in Bor(St(\mathbb{A}))$. It follows that in the proof we do not need the \textit{full} Grothendieck property of the algebra $\mathbb{A}$, but only the property for non-negative measures, i.e. if a sequence of non-negative measures on $\mathbb{A}$ is weakly* convergent, then it is weakly convergent. Such a variant of the Grothendieck property was introduced in Koszmider and Shelah \cite{KoszShel}, where it was called \textit{the positive Grothendieck property}. 
%
%

\begin{cor}\label{positive_grothendieck}
Let $\mathbb{A}$ be a Boolean algebra with the positive Grothendieck property and $(\mathbb{B},\mu)$ be a metric Boolean algebra. Let $(\varphi_n)$ be a sequence in $\mathcal{H}(\mathbb{A},\mathbb{B})$ pointwise metric convergent to a homomorphism $\varphi\in\mathcal{H}(\mathbb{A},\mathbb{B})$. Then, $(\varphi_n)$ converges pointwise Borel metric to $\varphi$.
\end{cor}

The following weaker form of the $\sigma$-completeness was also introduced in \cite{KoszShel}: a Boolean algebra $\mathbb{A}$ is said to have \textit{the Weak Subsequential Separation Property} (\textit{the WSSP}) if for every sequence $(A_n)$ of pairwise disjoint elements of $\mathbb{A}$ there is an element $A\in\mathbb{A}$ such that both of the sets $\big\{n\in\omega\colon A_n\le A\big\}$ and $\big\{n\in\omega\colon A_n\wedge A=0\big\}$ are infinite. \cite[Proposition 2.4]{KoszShel} asserts that every Boolean algebra with the WSSP has the positive Grothendieck property. Note that there exists a Boolean algebra with the WSSP (and hence with the positive Grothendieck property), but without the Grothendieck property, see \cite[Proposition 2.5]{KoszShel}.

Recall that by the Eberlein--\v{S}mulian theorem and the Dieudonn\'e--Grothendieck characterization of weakly compact sets of Radon measures on compact spaces (see \cite[Chapter VII, Theorem 14]{Diestel}), a weakly* convergent sequence $(\mu_n)$ of measures on a given compact space $K$ is weakly convergent if and only if there is no antichain $(U_k)$ of open subsets of $K$, $\eps>0$ and subsequence $(\mu_{n_k})$ such that $|\mu_{n_k}(U_k)|\ge\eps$ for every $k\in\omega$. 

\begin{prop}
Let $\mathbb{A}$ be a Boolean algebra with the positive Grothendieck property, $(\mathbb{B},\mu)$ a metric Boolean algebra, and $\varphi,\varphi_n\in\mathcal{H}(\mathbb{A},\mathbb{B})$, $n\in\omega$. Let $(A_n)$ be an antichain in $\mathbb{A}$ such that for some $\eps>0$ and every $n\in\omega$ we have $d_\mu\big(\varphi_n(A_n),\varphi(A_n)\big)\ge\eps$. Then, $(\varphi_n)$ is not pointwise metric convergent to $\varphi$.
\end{prop}
\begin{proof}
For the sake of contradiction assume that $(\varphi_n)$ is pointwise metric convergent to $\varphi$. Then, by Corollary \ref{convergences_measures_implications}.(3), the sequence $\big(\widehat{\mu}\circ f^{-1}_{\varphi_n}\big)$ is weakly* convergent to $\widehat{\mu}\circ f^{-1}_\varphi$. Since $(A_n)$ is an antichain in $\mathbb{A}$, the sequence $(\varphi(A_n))$ is an antichain in $\mathbb{B}$ and thus $\lim_{n\to\infty}\mu(\varphi(A_n))=0$. From the assumption that $d_\mu\big(\varphi_n(A_n),\varphi(A_n)\big)\ge\eps$ for every $n\in\omega$ it follows that there is $N\in\omega$ such that $\mu(\varphi_n(A_n))\ge\eps/2$ for every $n>N$, so $\big(\widehat{\mu}\circ f_{\varphi_n}^{-1}\big)(A_n)\ge\eps/2$ for every $n>N$. By the aforementioned characterisation of weakly convergent sequences of measures on compact spaces, the sequence $\big(\widehat{\mu}\circ f^{-1}_{\varphi_n}\big)$ is not weakly convergent to $\widehat{\mu}\circ f_\varphi^{-1}$, and hence, by the positive Grothendieck property of $\mathbb{A}$, not weakly* convergent, which is a contradiction.
\end{proof}

\section{Examples of sequences of homomorphisms}\label{sec-examples}

In this section we will present certain examples distinguishing different types of convergence. First, we summarize in the following self-explanatory diagram all the implications which we have already proved or which we are going to prove.\\[5pt]


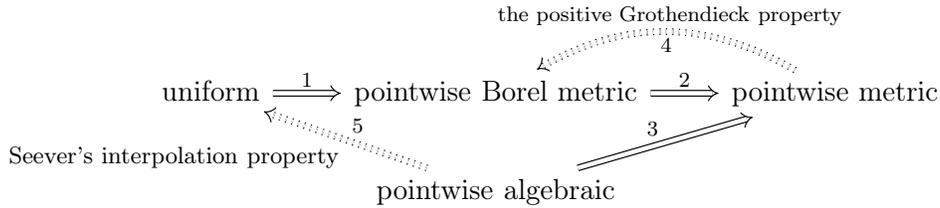
\begin{center}
	\begin{tikzcd}[arrows=Rightarrow]
		\mbox{uniform}  \arrow[r, "1"]   & \mbox{pointwise Borel metric} \arrow[r, "2"]  & \arrow[l, dotted,  "\mbox{\scriptsize{ the positive Grothendieck property}}" swap, "4", bend right] \mbox{pointwise metric}  \\
		& \arrow[lu, dotted, "\mbox{ \footnotesize{Seever's interpolation property}}", "5" swap] \mbox{pointwise algebraic} \arrow[ru, "3"] &  
\end{tikzcd}
\end{center}
\label{diagram}
.\\[5pt]

Of those, (2) is trivial, (1) follows from Corollary \ref{uniform_metrics_convergence} and (3) is proved in Proposition \ref{algebraic_metric}. No other implication holds in general. (4) holds for Boolean algebras with the positive Grothendieck
property and hence, in particular, for Boolean algebras with the WSSP (Corollary \ref{positive_grothendieck}). (5) is true for Boolean algebras with the interpolation property (or, the property (I)) introduced by Seever (see Section \ref{secEfimov}). Corollary \ref{alg_pointwise_coincide} implies that (3) may be reversed for $\mathbb{B}=\mathbbm{2}$, but we skipped that in the diagram, since it is trivial. 

Note that it is easy to obtain a sequence which satisfies simultaneously all of the above types of convergence (e.g. a trivial sequence). Below we present examples of sequences of homomorphisms satisfying various other combinations of types of convergence and witnessing that the above diagram is complete.

In what follows we will often abuse the notation identifying equivalence classes (particularly elements of $\mathbb{M}$) with their representatives.

\subsection{A pointwise algebraically convergent sequence which is not pointwise Borel metric convergent}\label{ex_alg_not_borel}

To see that the pointwise algebraic convergence does not imply the pointwise Borel metric convergence, consider the Cantor algebra $\mathbb{C}$ and a non-trivial sequence $(x_n)$ of points in the Cantor space $St(\mathbb{C})=2^\omega$ convergent to some $x\in 2^\omega$. For each
$n\in\omega$ define $\varphi_n\in\mathcal{H}(\mathbb{C},\mathbbm{2})$ by the condition: $\varphi_n(A) = 1$ iff $x_n\in A$, where $A\in\mathbb{C}$. Similarly, define $\varphi\in\mathcal{H}(\mathbb{C},\mathbbm{2})$ by the condition: $\varphi(A) = 1$ iff $x\in A$, where $A\in\mathbb{C}$. Since $\lim_{n\to\infty}x_n=x$ in $St(\mathbb{C})$, $(\varphi_n)$ is pointwise algebraically convergent. On the other hand, $(\varphi_n)$ does not converge pointwise Borel metric to $\varphi$, since for the Borel set $\{x\}$ and every $n\in\omega$ we have:
\[\widehat{\lambda}_1\big(f_{\varphi_n}^{-1}\big[\{x\}\big]\triangle f_{\varphi}^{-1}\big[\{x\}\big]\big)=\lambda_1(0\triangle 1)=1.\]

Of course, in the above example instead of $\mathbb{C}$ we may use any Boolean algebra such that its Stone space contains a non-trivial convergent sequence.

\subsection{A uniformly convergent sequence which is not pointwise algebraically convergent}\label{ex_unif_not_alg}

Here we will present an example of a sequence of homomorphisms in $\mathcal{H}(\mathbb{C},\mathbb{M})$ which converges uniformly but not pointwise algebraically.

Let $\{s_n\colon n\in \omega\}$ be any enumeration of $2^{<\omega}$. For $x\in 2^\omega$ define $\overline{x} \in 2^\omega$ by 
$$
\overline{x}(k) = 
\begin{cases}
	1-x(k), \mbox{ if }k=0,\\
	x(k), \mbox{ otherwise,}
\end{cases}
$$

so $\overline{x}$ is just $x$ but with the $0$-th bit flipped. For $s\in 2^{<\omega}$ and $x\in 2^\omega$ we will say that $x$ \emph{agrees with} $s$, if $x(k+1) = s(k)$ for each $k< |s|$
(so, in other words, $s$ is a segment of $x$ starting with the $1$-st bit).

Now for $n\in \omega$ define $g_n\colon 2^\omega \to 2^\omega$ by
$$
g_n(x) = 
\begin{cases}
	\overline{x}, \mbox{ if } x \mbox{ agrees with }s_n,\\
	x, \mbox{ otherwise.}
\end{cases}
$$

It is easy to check that $\psi_n\colon \mathbb{C} \to \mathbb{M}$ defined by $\psi_n(A) = g_n[A] = g_n^{-1}[A]$ is a homomorphism. Let $\psi$ be the natural embedding of $\mathbb{C}$ into $\mathbb{M}$, i.e. $\psi(A) = A$ for every $A\in\mathbb{C}$. (Note that here we treat $\mathbb{C}$ as the algebra of clopen subsets of $2^\omega$ and writing $g_n[A]$ we mean the appropriate equivalence class, that is, an element of $\mathbb{M}$.)

\begin{prop} $(\psi_n)$ converges uniformly to $\psi$ but it does not converge pointwise algebraically.
\end{prop}

\begin{proof}
	For each $s\in 2^{<\omega}$ let $A_s = \{x\in 2^\omega\colon x\mbox{ agrees with }s\}$. Clearly, $\lambda(A_s) = 1/2^{|s|}$ and so $\lim_{n\to\infty}\lambda(A_{s_n})=0$. But $\psi_n$ coincides with $\psi$ on $2^\omega \setminus A_{s_n}$ and so $(\psi_n)$
	converges uniformly to $\psi$.

	Let $A=\{x\in 2^\omega\colon x(0)=0\}$. 
	Fix $x\in 2^\omega$ and $n\in\omega$. If $x\in A$, then there is $m>n$ such that $x$ agrees with $s_m$ and so $x\not\in g_m[A]$. Hence, $A\cap\bigcap_{m>n}g_m[A]=\emptyset$. If $x\not\in A$, then there is $m>n$ such that $x$ does not agree with $s_m$ and so $x\not\in g_m[A]$. Hence, $A^c\cap\bigcap_{m>n}g_m[A]=\emptyset$, and so $\bigcap_{m>n}g_m[A]=\emptyset$. Since $n$ was arbitrary, we get that $\bigcup_n\bigcap_{m>n}g_m[A]=\emptyset$ and hence $\bigvee_n \bigwedge_{m>n} \psi_n(A) = 0 \ne A=\psi(A)$, which implies that $(\varphi_n)$ does not converge pointwise algebraically to $\psi$.
\end{proof}

Note that since the pointwise algebraic convergence implies the pointwise metric convergence, it is not possible that $(\psi_n)$ converges pointwise algebraically to a homomorphism other than $\psi$.

\subsection{A pointwise metric convergent sequence which is neither pointwise Borel metric convergent nor pointwise algebraically convergent}\label{ex_metric_not_borel_alg}

Combining methods used in Sections \ref{ex_alg_not_borel} and \ref{ex_unif_not_alg}, we may easily obtain an example of a sequence of homomorphisms in $\mathcal{H}(\mathbb{C},\mathbb{M})$ 
which is only pointwise metric convergent. 

For each $n\in\omega$, let $\varphi_n \colon \mathbb{C} \to \mathbbm{2}$ be defined as in Section \ref{ex_alg_not_borel} and let $\psi_n\colon \mathbb{C} \to \mathbb{M}$ be defined as in Section \ref{ex_unif_not_alg}. Define $\rho_n\colon \mathbb{C} \to \mathbbm{2}\times \mathbb{M}$ by the formula $\rho_n(A) = (\varphi_n(A),\psi_n(A))$ for every $A\in\mathbb{C}$. 

By the same arguments as in Sections \ref{ex_alg_not_borel} and \ref{ex_unif_not_alg}, the sequence $(\rho_n)$ is pointwise metric convergent to the homomorphism $\rho=(\varphi,\psi)$, yet, again by the same arguments, it is not pointwise Borel
metric convergent nor pointwise algebraically convergent. Of course, $\mathbbm{2}\times \mathbb{M}$ can be embedded into $\mathbb{M}$, so $(\rho_n)$ can be treated as a sequence in $\mathcal{H}(\mathbb{C},\mathbb{M})$.

\subsection{A pointwise Borel metric convergent sequence which is neither uniformly convergent nor pointwise algebraically convergent}\label{ex_borel_not_unif_alg}



For every $n\in\omega$ define $flip_n\colon 2^\omega \to 2^\omega$ by $flip_n(x) = x + \chi_{\{n\}}$, where $+$ is pointwise modulo $2$ (so, in short, $flip_n$ flips the $n$-th coordinate of $x$). In the analogous way we define $flip_n$ also on $2^{<\omega}\setminus 2^{<n}$
(of course, $flip_n(s)$ makes sense only if $n<|s|$). 
Let $\varphi_n\colon \mathbb{M} \to \mathbb{M}$ be the
	homomorphism induced by $flip_n$, i.e. the homomorphism defined by $\varphi_n(B) = flip_n[B]$. 
	Then $(\varphi_n)$ is convergent to the identity on $\mathbb{M}$ in the pointwise metric topology. Indeed, first notice that $(\varphi_n(A))$ converges trivially to $A$ if $A$ is a clopen subset of $2^\omega$. Also, for every $\varepsilon>0$ and every measurable set $B$ there is a clopen set $A$ such that $d_\lambda(B,A)<\varepsilon$. Since $\varphi_n$ is an isometry for each $n$, we get that $\varphi_n(B)$ converges with respect to the pointwise metric topology for every Borel subset $B$ of $2^\omega$.  

	As $\mathbb{M}$ has the Grothendieck property, $(\varphi_n)$ is automatically pointwise Borel metric convergent.
 On the other hand, $d_\lambda\big( \varphi_n(X_n), X_n\big)=1$ for each $n$, where $X_n = \{x\in 2^\omega\colon x(n)=0\}$, so $(\varphi_n)$ does not converge uniformly.

	The rest of this section will be devoted to the proof that $(\varphi_n)$ does not converge pointwise algebraically. The crucial fact is
	the following.

\begin{prop}\label{strange} There is a closed subset $B$ of $2^\omega$ such that $\lambda(B)>1/2$ and for each $x\in B$ we have $flip_n(x)\notin B$ for infinitely many $n$'s.
\end{prop}

Perhaps quite unexpectedly, to prove it we will use several notions from coding theory. For $n\in \omega$ the function $\rho_n\colon \{0,1\}^n \to \omega$ defined by
\[\rho_n(x,y) = \big|\{i\leq n\colon x(i)\ne y(i)\}\big|\] 
is clearly a metric (called  
\emph{the Hamming metric}). We will call a subset $A\subseteq \{0,1\}^n$ \emph{a code} if for every $x\in \{0,1\}^n$ there is $y\in A$ such that $\rho_n(x,y)\leq 1$. In coding theory this property is one of the conditions defining so-called 1-error correcting
codes.

In what follows we will use the fact that there is a perfect sphere packing in $\{0,1\}^n$ for each $n$ of the form $n=2^m-1$, $m\in \omega$, i.e. $\{0,1\}^n$ is a disjoint union of the closed balls of radius $1$ with respect to the Hamming metric (see \cite{Hamming}). The set of centers of those
balls is a code, which we will call \emph{a perfect code} (in the coding theory terminology this is the same as a perfect 1-error-correcting Hamming code).
There is a rich literature concerning Hamming codes and it is quite easy to learn \emph{how} to find perfect codes. It seems, however, that it is not so easy to find a proof \emph{why} those algorithms work, so we sketch here the argument, using
the algorithm for creating perfect codes presented e.g. in \cite[Section 2.5]{Designs}. 

\begin{prop}\label{perfect_code_exists} If $n  = 2^m -1$, for $m\in \mathbb{N}$, then $\{0,1\}^n$ is a disjoint union of the closed balls of radius $1$ with respect to the Hamming metric.
\end{prop}

\begin{proof}
	Let $n = 2^m - 1$. We endow the set $\{0,1\}^n$ with the algebraic structure and see it as $\mathbb{F}^n_2$ (i.e. $n$-dimensional vector space over $\mathbb{F}_2$, the field consisting of two elements). Let $H$ be a matrix created by putting all the
	non-zero elements of $\mathbb{F}^m_2$ as its columns. Then $H$ is an $(n\times m)$-matrix. Define 
	\[ C = \big\{ v\in\mathbb{F}^n_2\colon Hv = 0_{\mathbb{F}^n_2} \big\}. \]
	In other words, $v\in C$ if and only if $\sum_{l<n} H(k,l) \cdot v(l) = 0$ for each $k<m$.

	We claim that $C$ is a perfect code, i.e. the collection $\mathcal{K} = \{K_1(c)\colon c\in C\}$, where $K_r(x)$ is a closed ball of radius $r$ centered at $x$, forms a partition of $\mathbb{F}^n_2$.

First, notice that if $v \in C$ and $1 \leq \rho_n(v, w) \leq 2$ for some $w\in\mathbb{F}_2^n$, then $w \notin C$. Indeed, suppose first that there is exactly one $i<n$ such that $v(i) \ne w(i)$. There is $k< m$ such that $H(k,i)=1$ and then $\sum_{l<n} H(k,l)\cdot
w(l) = 1$ and so
$w\notin C$. If there are exactly two indices $i<j< n$ such that $v(i) \ne w(i)$ and $v(j) \ne w(j)$, then for $k<m$ such that $H(k,i) \ne H(k,j)$ we again have $\sum_{l<n} H(k,l)
\cdot w(l) = 1$, and $w\notin C$.

So, by the above, the family $\mathcal{K}$ is pairwise disjoint. Furthermore, we may arrange columns of $H$ as we wish, so we may suppose that the last $m$ columns of $H$ consist of vectors with only one $1$. In this way it
is easy to see that every vector in $\mathbb{F}_2^{n-m}$ has exactly one extension to the vector in $\mathbb{F}_2^n$ which belongs to $C$ and so $|C| = 2^{n-m}$. As every element of $\mathcal{K}$ consists of $n+1$ points and every point is covered by
at most one element of $\mathcal{K}$, we see that $\mathcal{K}$
covers $(n+1) \cdot 2^{n-m} = 2^n$ points, i.e. the whole space $\mathbb{F}_2^n$. 
\end{proof}

Let $c_n$ be
the smallest size of a code in $\{0,1\}^n$ and let $a_n = c_n/2^n$ (so $a_n$ measures how big portion of whole $\{0,1\}^n$ we have to take to obtain a code). The following fact is a consequence of the existence of perfect codes.

\begin{prop}\label{Hamming} $\lim_{n\to\infty} a_n = 0$.
\end{prop}

\begin{proof}
First, notice that $(a_n)$ is non-increasing. Indeed, if $A$ is a code in $\{0,1\}^n$, then the family $A' = \{x^\frown 0\colon x\in A\} \cup \{x^\frown 1\colon x\in A\}$ is a code in $\{0,1\}^{n+1}$ and $|A'| = 2|A|$.

To show that $(a_n)$ can be arbitrarily close to $0$, fix $m\in \omega$ and let $n = 2^m -1$. By Proposition \ref{perfect_code_exists}, there is a perfect code $C$ in $\{0,1\}^n$, and from the proof we know that $|C| = 2^{n-m} = 2^n/(n+1)$ and so $c_n \leq 2^n/(n+1)$ and hence $a_n \leq 1/(n+1)$. As $(a_n)$ is non-increasing, $\lim_{n\to \infty} a_n = 0$.
\end{proof}

%

\begin{proof}[\textbf{Proof of Proposition \ref{strange}}] We are going to construct inductively an increasing sequence $(d_n)$ of natural numbers and a
	sequence $(C_n)$ of sets such that for every $n\in \omega$ the following conditions hold:
\begin{enumerate}
	\item $C_n \subseteq \{0,1\}^{d_n}$,
	\item $|C_n|<2^{d_n - 1}$,
	\item for every $c\in C_n$ and $t\in\{0,1\}^{d_{n+1}-d_n}$ we have $c^\frown t\in C_{n+1}$,
	\item For each $s \in \{0,1\}^{d_{n+1}}\setminus C_{n+1}$ there is $d_n\le m<d_{n+1}$ such that $flip_m(s)\in C_{n+1}$. 
\end{enumerate}

Let $d_0 = 3$ and $C_0$ be a perfect code in $\{0,1\}^3$ (e.g. $C_0$ may consist of two constant sequences). Suppose now that we have constructed $d_n$ and $C_n$. Let
\[ \alpha = \frac{2^{d_n-1}-|C_n|}{2^{d_n}-|C_n|}. \]
By (2) $\alpha>0$ and so, using Proposition \ref{Hamming}, we may find $k$ so that $a_{k} < \alpha$.
Let $C$ be a code in $\{0,1\}^{k}$ such that $|C| = a_k2^k$. Define $C_{n+1}$ in the following way:
\[ C_{n+1} = \big\{t^\frown x\colon t\in C_n, x\in \{0,1\}^k\big\} \cup \big\{t^\frown c\colon t\in \{0,1\}^{d_n}\setminus C_n,c\in C\big\}. \]
Thus, $C_{n+1}$ consists of two parts: all the possible extensions of elements of $C_n$ to $\{0,1\}^{d_n+k}$ and all the extensions of the rest of $\{0,1\}^{d_n}$ by elements of $C$. Put $d_{n+1} = d_n + k$. We claim that $C_{n+1}$ and $d_{n+1}$ satisfy all the desired properties. Indeed, (1) and (3) are clear, and concerning (2) we have:
\[ |C_{n+1}| = 2^k|C_n| + (2^{d_n}-|C_n|)\cdot a_k2^k, \] 
so
\begin{multline*}
	\frac{|C_{n+1}|}{2^{d_{n+1}}} = \frac{|C_{n+1}|}{2^{d_n+k}} = \frac{|C_n|}{2^{d_n}} + \frac{(2^{d_n}-|C_n|)\cdot a_k}{2^{d_n}} < \\ < \frac{|C_n|}{2^{d_n}} + \frac{(2^{d_n}-|C_n|)}{2^{d_n}}\cdot \frac{(2^{d_n-1}-|C_n|)}{(2^{d_n}-|C_n|)} = \frac{2^{d_n-1}}{2^{d_n}} = \frac{1}{2}, 
\end{multline*}
hence $|C_{n+1}|<2^{d_{n+1}-1}$.

Finally, to check (4) let $s\in \{0,1\}^{d_{n+1}}\setminus C_{n+1}$ and $t = s\restriction d_n$. If $t\in C_n$, then $flip_m(s)\in C_{n+1}$ for every $d_n\le m<d_{n+1}$. If $t\not\in C_n$, then, as $C$ is a code, there is $c\in C$ such that $\rho_{d_{n+1}}(t^\frown c,s)\leq 1$, which means that there is $d_n\le m<d_{n+1}$ 
 such that $flip_m(s) = t^\frown c$ and so $flip_m(s) \in C_{n+1}$.


\medskip

Let $B_n = 2^\omega \setminus \bigcup_{c\in C_n} [c]$. By (3), 
 the sequence $(B_n)$ is $\subseteq$-decreasing. Let $B = \bigcap_n B_n$. Because
of (2) we have that $\lambda(B_n)>1/2$ for each $n$ and so $\lambda(B)\geq 1/2$. 

Let $x\in 2^\omega$. If $x\notin B$, then $c = x\restriction d_n \in C_n$ for some $n$. Then $flip_m(x) \in [c]$ for every $m>d_n$ and so $flip_m(x)\notin B$ for every $m>d_n$. Now, assume that $x\in B$, so $x\restriction d_n\not\in C_n$ for any $n\in\omega$. Fix $N\in\omega$. Find $n$ such that $d_n > N$ and let $s = x\restriction
	d_{n+1}$, so that $s\in\{0,1\}^{d_n+1}\setminus C_{n+1}$.  By (4), there is $m\geq d_n$ such that $flip_m(s)\in C_{n+1}$. This means that $flip_m(x) \notin B_{n+1}$ and so $flip_m(x)\notin B$. Since $N$ was arbitrary, we are done.
\end{proof}

\begin{cor} $(\varphi_n)$ does not converge pointwise algebraically.
\end{cor}

\begin{proof}
	Let $B$ be a set given by Proposition \ref{strange}. For each $m$ we have
	\[ \bigcap_{n>m} flip_n[B] = \emptyset. \]
	Indeed, if $x\in \bigcap_{n>m} flip_n[B]$, then for every $n>m$ there is $y_n\in B$ such that $x = flip_n(y_n)$. But this means that $flip_n(x)=y_n\in B$ for each $n>m$, which is a contradiction with the property of $B$ promised by Proposition \ref{strange}.

	It follows that $\bigvee_m \bigwedge_{n>m} \varphi_n(B) = 0$, but $B\ne 0$ and so $(\varphi_n)$ does not converge pointwise algebraically to the identity $\psi$.
\end{proof}

\begin{rem}\label{>1/2} Notice that in fact in Proposition \ref{strange} we can ask for $B$ of arbitrarily large measure (smaller than $1$)---it is enough to adjust the definition of $\alpha$ in the proof.
\end{rem}

\begin{rem} Using the Vitali equivalence relation, it is easy to construct a \emph{non-measurable} set $A$ such that $\lambda^*(A)=1$ with even stronger property: for each $x \in A$ and every $n\in \omega$ we have $flip_n(x)\notin A$. 
	Using this set is probably the easiest method of showing the existence of a infinite game with perfect information which is not determined (see e.g. \cite{Kopczynski}).
It seems however that to obtain a Borel set, we have to be much more careful and we do not know any way to do this without using Hamming codes.
\end{rem}

\subsection{A sequence which is pointwise Borel metric convergent and pointwise algebraically convergent but not uniformly convergent}\label{ex_borel_alg_not_unif}

The above example may be adapted to obtain a sequence which is pointwise Borel metric convergent and pointwise algebraically convergent but not uniformly convergent.

Let $(\varphi_n)$ be the sequence defined in Section \ref{ex_borel_not_unif_alg}. Define $\psi_n\colon \mathbb{C} \to \mathbb{M}$ by $\psi_n = \varphi_n \restriction \mathbb{C}$. Here, we treat $\mathbb{C}$ as a subalgebra of $\mathbb{M}$, identifying clopen
subsets of $2^\omega$ with its equivalence relations (i.e. the elements of $\mathbb{M}$). By the same argument which showed that $(\varphi_n)$ was pointwise metric convergent to the identity on $\mathbb{M}$, $(\psi_n)$ is pointwise \textit{Borel} metric convergent to the identity on $\mathbb{C}$. Since the sets $X_n$'s from Section \ref{ex_borel_not_unif_alg} are clopen in $2^\omega$, $(\psi_n)$ is not uniformly convergent. It is however pointwise algebraically convergent, since $flip_n[C]$ is eventually constant for $C$ being a clopen subset of $2^\omega$.

\medskip

In Section \ref{secEfimov} we prove however that if a Boolean algebra $\mathbb{A}$ is somewhat more complicated than $\mathbb{C}$, e.g. $\mathbb{A}$ is $\sigma$-complete, then there is no similar example in $\mathcal{H}(\mathbb{A},\mathbb{M}_\kappa)$.

\section{Homomorphisms and forcing names for ultrafilters}\label{sec-homo-names}


Fix, once and for all, an infinite cardinal $\kappa$ and suppose that we force with the notion $\mathbb{M}_\kappa$ over a model $V$ of set theory. If $\mathbb{A}$ is a Boolean algebra in $V$ and $G$ is an $\mathbb{M}_\kappa$-generic filter over $V$, then we can
consider $\mathbb{A}$ as an element of $V[G]$, however, there might be new ultrafilters on $\mathbb{A}$ in $V[G]$. It appears nevertheless that to each $\mathbb{M}_\kappa$-name $\dot{\mathcal{U}}$ for an ultrafilter, i.e. to such $\dot{\mathcal{U}}$ that $\Vdash_{\mathbb{M}_\kappa}``\dot{\mathcal{U}}$ is an ultrafilter on $\mathbb{A}$'', we can assign a homomorphism from $\mathbb{A}$ to $\mathbb{M}_\kappa$ in $V$ in a quite natural way. Namely, if $\dot{\mathcal{U}}$ is such an $\mathbb{M}_\kappa$-name, then the function $\varphi_{\dot{\mathcal{U}}}\colon \mathbb{A} \to \mathbb{M}_\kappa$ defined fo every $A\in\mathbb{A}$ (recall that $\mathbb{A}\in V$) by the formula 
\[\varphi_{\dot{\mathcal{U}}}(A) = \llbracket A\in \dot{\mathcal{U}}\rrbracket\] 
is a Boolean homomorphism belonging to $V$. (We skip writting the symbol $\check{ }$ over canonical names for elements of $V$.)

Now fix a homomorphism $\varphi\colon \mathbb{A} \to \mathbb{M}_\kappa$ in $V$ and define an $\mathbb{M}_\kappa$-name $\tau_\varphi$ for a family of subsets $\mathbb{A}$ by
\[ \tau_\varphi=\big\{\langle A, \varphi(A)\rangle\colon A \in \mathbb{A}\big\}. \]
Then, $\Vdash_{\mathbb{M}_\kappa}$``$\tau_\varphi$ is an ultrafilter on $\mathbb{A}$''; in fact, we actually get that $\Vdash_{\mathbb{M}_\kappa}\tau_\varphi=\varphi^{-1}[\Gamma]$, where $\Gamma$ is the canonical $\mathbb{M}_\kappa$-name for an $\mathbb{M}_\kappa$-generic filter over $V$.

\begin{lem}\label{hom_uf_ass}
	For every $\mathbb{M}_\kappa$-name $\dot{\mathcal{U}}$ for an ultrafilter on a Boolean algebra $\mathbb{A}$ and for every homomorphism $\varphi\colon\mathbb{A}\to\mathbb{M}_\kappa$ it holds $\varphi_{(\tau_\varphi)}=\varphi$ and $\Vdash_{\mathbb{M}_\kappa}
	\tau_{(\varphi_{\dot{\mathcal{U}}})}=\dot{\mathcal{U}}$.
\end{lem}

In this way we obtain a bijective correspondence between elements of $\mathcal{H}(\mathbb{A},\mathbb{M}_\kappa)$ and  ultrafilters on $\mathbb{A}$ in the forcing extension by $\mathbb{M}_\kappa$. One of the advantages of such an approach is that we may use properties of measures associated to homomorphisms (as in Section \ref{sec-convergence}) to study $\mathbb{M}_\kappa$-names for ultrafilters on \emph{old} Boolean algebras (in particular, to study ultrafilters
on the Cantor algebra, i.e. reals). 

\begin{exa} Suppose that $\varphi\colon \mathbb{A} \to \mathbb{M}_\kappa$ is such that the measure $\lambda_\kappa \circ \varphi$ is a point mass measure $\delta_\mathcal{V}$ for some $\mathcal{V}\in St(\mathbb{A})\cap V$. Then, $\Vdash_{\mathbb{M}_\kappa} \tau_\varphi = \mathcal{V}$. 
\end{exa}

A study of the interplay between the properties of $\mathbb{M}_\kappa$-names for ultrafilters and measures is a subject of forthcoming paper \cite{Cegielka}.

%

\textbf{In what follows} $(\varphi_n)$ and $\varphi$ will denote (a sequence of) elements of $\mathcal{H}(\mathbb{A},\mathbb{M}_\kappa)$ associated---as described above---to (a sequence of) $\mathbb{M}_\kappa$-names $(\dot{\mathcal{U}}_n)$ and $\dot{\mathcal{U}}$ for ultrafilters on a fixed Boolean algebra $\mathbb{A}\in V$, respectively. Note that by Lemma \ref{hom_uf_ass}, it has no meaning whether we associate homomorphisms in $\mathcal{H}(\mathbb{A},\mathbb{M}_\kappa)$ to $\mathbb{M}_\kappa$-names, or $\mathbb{M}_\kappa$-names to homomorphisms in $\mathcal{H}(\mathbb{A},\mathbb{M}_\kappa)$.

Recall here that for every sequence $(\phi_n)$ of formulae the following equalities hold:
\[\llbracket\exists^\infty n\colon\ \phi_n\rrbracket=\llbracket\forall n\exists m>n\ \phi_n\rrbracket=\bigwedge_n\bigvee_{m>n}\llbracket\phi_n\rrbracket\]
and
\[\llbracket\forall^\infty n\colon\ \phi_n\rrbracket=\llbracket\exists n\forall m>n\ \phi_n\rrbracket=\bigvee_n\bigwedge_{m>n}\llbracket\phi_n\rrbracket;\]
we will use them frequently.

\subsection{Pointwise algebraic convergence}\label{alg->forcing}

There is a strong connection between the pointwise algebraic convergence of sequences of homomorphisms from a Boolean algebra $\mathbb{A}$ into $\mathbb{M}_\kappa$ in the ground model $V$ and the convergence of the sequences of corresponding ultrafilters in the Stone space $St(\mathbb{A})$ in the $\mathbb{M}_\kappa$-generic extension $V[G]$.

\begin{prop}\label{convergence} The following conditions are equivalent:
	\begin{itemize}
		\item $\Vdash_{\mathbb{M}_\kappa}``(\dot{\mathcal{U}}_n)$ converges to  $\dot{\mathcal{U}}$'',
		\item $(\varphi_n)$ converges to $\varphi$ pointwise algebraically.
	\end{itemize}
\end{prop}

\begin{proof} 
	Suppose that it is not true that  $\Vdash_{\mathbb{M}_\kappa}``(\dot{\mathcal{U}}_n)$ converges to  $\dot{\mathcal{U}}$''. It means that there is a non-zero $p\in \mathbb{M}_\kappa$ which forces the opposite, i.e. that there is $A\in \mathbb{A}$ such that $p \Vdash
	\exists^\infty n\  A\in \dot{\mathcal{U}}\setminus \dot{\mathcal{U}}_n$. It follows that
\[p \leq \bigwedge_m \bigvee_{n>m} \llbracket A\in \dot{\mathcal{U}}\setminus \dot{\mathcal{U}}_n \rrbracket,\] 
and hence
\[\bigwedge_m \bigvee_{n>m}\big(\varphi(A) \setminus	\varphi_n(A)\big) \ne 0.\] 
Thus, $(\varphi_n(A))$ does not converge to $\varphi(A)$ and hence $(\varphi_n)$ does not converge pointwise algebraically to $\varphi$.

If $(\varphi_n(A))$ does not converge algebraically to $\varphi(A)$ for some $A\in \mathbb{A}$, then 
\[p=\varphi(A) \setminus\big(\bigwedge_m \bigvee_{n>m}\varphi_n(A)\big)\ne 0\quad\text{or}\quad q=\big(\bigwedge_m \bigvee_{n>m}\varphi_n(A)\big)\setminus \varphi(A) \ne 0.\]

Assume first that the former case holds, i.e. $p\neq0$. We have:
\[p=\varphi(A)\wedge\big(\bigvee_m \bigwedge_{n>m}\varphi_n(A^c)\big),\]
so $p\Vdash A\in\dot{\mathcal{U}}$ and $p\Vdash\exists^\infty n\ A^c\in\dot{\mathcal{U}}_n$. It follows that $p\Vdash``(\dot{\mathcal{U}}_n)$ does not converge to $\dot{\mathcal{U}}$''. 

We proceed similarly in the latter case, i.e. when $q\neq\emptyset$.

\end{proof}

\subsection{Uniform convergence}

In this section we will show that the notion of uniform convergence of homomorphisms is connected to trivial convergence of ultrafilters in the random extension.



\begin{prop} \label{uniform-convergence}
	If $\Vdash_{\mathbb{M}_\kappa} \forall^\infty n\in\omega\ \dot{\mathcal{U}}_n = \dot{\mathcal{U}}$, then $(\varphi_n)$ converges uniformly to $\varphi$. 
\end{prop}

\begin{proof}
 Assume that $(\varphi_n)$ does not converge uniformly to $\varphi$. Then, without loss of generality, there is $\varepsilon>0$ such that for each $n\in \omega$ there is $A_n\in\mathbb{A}$ such that  $\lambda\big(\varphi(A_n)
	\setminus \varphi_n(A_n)\big)>\varepsilon$. For every $n\in\omega$ put $p_n=\varphi(A_n)\setminus \varphi_n(A_n)$. Since $\lambda\big(\bigvee_{n>m}p_n\big)>\varepsilon$ for every $m\in\omega$, for $p = \bigwedge_{m} \bigvee_{n>m} p_n$ we have $\lambda(p)\ge\varepsilon$ and thus $p\neq 0$. 

	Let $G$ be an $\mathbb{M}_\kappa$-generic filter over $V$ containing $p$. We work in $V$. There is  $q\in G$ below $p$ such that $q\Vdash_{\mathbb{M}_\kappa} \forall^\infty n\in\omega\ \dot{\mathcal{U}}_n = \dot{\mathcal{U}}$, and hence there is $r\in G$ below $q$ and $m\in\omega$ such that for every $n>m$ we have $r\Vdash\dot{\mathcal{U}}_n=\dot{\mathcal{U}}$. It follows immediately that $p_n\perp r$ for every $n>m$ and thus $p\perp r$, a contradiction, since $r\leq p$.
\end{proof}

It appears that a fact which is in a sense converse to Proposition \ref{uniform-convergence} also holds and constitutes actually the main result of this section (Theorem \ref{uniform-convergence2}). Note here that in the proof of Proposition \ref{uniform-convergence}, as well as of Proposition \ref{convergence}, we actually did not use that the Boolean algebra $\mathbb{M}_\kappa$ carries a strictly positive measure and thus those results remain in fact true for any $\sigma$-complete Boolean algebra $\mathbb{B}$ in place of $\mathbb{M}_\kappa$ (and hence, of course, for any reasonable notion of forcing).

\begin{thm}\label{uniform-convergence2} If $(\varphi_n)$ converges uniformly to $\varphi$, then for almost all $n$ there is $p_n \in \mathbb{M}_\kappa$ such that $p_n \Vdash \dot{\mathcal{U}}_n = \dot{\mathcal{U}}$.
\end{thm}

We will prove this theorem in a series of lemmas and propositions.  One of them, Theorem \ref{distinct}, is interesting on its own sake.

\begin{thm}\label{distinct}
	Let $\mathbb{A}$ be an atomless Boolean algebra in $V$. If $\dot{\mathcal{U}}$ and $\dot{\mathcal{V}}$ are $\mathbb{M}_\kappa$-names for ultrafilters on $\mathbb{A}$ such that $\Vdash_{\mathbb{M}_\kappa} \dot{\mathcal{U}} \ne\dot{\mathcal{V}}$, then for every
	$\varepsilon>0$ there is $p\in \mathbb{M}_\kappa$ and $C\in \mathbb{A}$ such that $\lambda(p)>1/4 - \varepsilon$ and $p\Vdash C\in \dot{\mathcal{U}} \triangle \dot{\mathcal{V}}$. 
\end{thm}

Theorem \ref{distinct} requires a brief comment. Assume that $\mathbb{A}$ is a Boolean algebra and $\dot{\mathcal{U}}$ and $\dot{\mathcal{V}}$ are $\mathbb{M}_\kappa$-names for ultrafilters on $\mathbb{A}$ such that $\Vdash_{\mathbb{M}_\kappa} \dot{\mathcal{U}} \ne \dot{\mathcal{V}}$. Since $\mathbb{M}_\kappa$ is a ccc forcing, there is a maximal
antichain $(p_n)$ in $\mathbb{M}_\kappa$ and a sequence $(A_n)$ of elements of $\mathbb{A}$ such that $p_n \Vdash A_n \in \dot{\mathcal{U}}\setminus \dot{\mathcal{V}}$. The problem with this general statement is that \emph{a priori} we do not have
any control on the measures of $p_n$'s---it might as lief happen that all of them are very small in terms of the measure. Theorem \ref{distinct} shows however that we can always find one ``large'' condition in $\mathbb{M}_\kappa$ which distinguishes the ultrafilters.

The proof of Theorem \ref{distinct} requires some auxiliary lemmas. 

\begin{prop}\label{malowanie-plotu-new}
Assume that $(X, \Sigma, \mu)$ is a measure space with a non-negative (not necessarily probability) measure $\mu$, $K$ is a compact space, and $f,g\colon X \to K$ are simple measurable functions. If $f \ne g$ $\mu$-almost everywhere, then there exists $L\in\Sigma$ such that $\mu(L)\ge \mu(X)/4$ and $f[L] \cap g[L] =\emptyset$.
\end{prop}
\begin{proof}
Let $\mathcal{P}$ be a finite partition of $X$ into measurable pieces such that for every $A\in\mathcal{P}$ the functions $f$ and $g$ are constant on $A$ and $\mu(A)>0$. Define the functions $f_{\mathcal{P}},g_{\mathcal{P}}\colon\mathcal{P}\to K$ as follows:
\[f_{\mathcal{P}}(A)=x\quad\text{if and only if}\quad f[A]=\{x\},\]
and similarly
\[g_{\mathcal{P}}(A)=x\quad\text{if and only if}\quad g[A]=\{x\}.\]
For every $A\in\mathcal{P}$ we have $f_{\mathcal{P}}(A)\neq g_{\mathcal{P}}(A)$ (since $\mu(A)>0$). Put $H=f[X]\cap g[X]$. Define an auxiliary measure $\nu$ on $\mathcal{P}(H\times H)$ as follows:
\[\nu\big(\{(x,y)\}\big)=\sum\big\{\mu(A)\colon\ f_{\mathcal{P}}(A)=x,\ g_{\mathcal{P}}(A)=y, A\in\mathcal{P}\big\}\]
for every $(x,y)\in H\times H$. It follows that $\nu\big(\{(x,x)\}\big)=0$ for every $x\in H$.

For every $N\subseteq H$ let
\[N_0=(N\times N)\cup(N^c\times N^c)\quad\text{and}\quad N_1=(N\times N^c)\cup(N^c\times N).\]
It follows that for every $(x,y)\in H\times H$ with $x\neq y$ we have:
\[\big|\{N\subseteq H\colon\ (x,y)\in N_0\}\big|=\big|\{N\subseteq H\colon\ (x,y)\in N_1\}\big|,\]
and so
\[\sum_{N\in\mathcal{P}(H)}\nu(N_0)=\sum_{N\in\mathcal{P}(H)}\nu(N_1).\]
There exists $N\subseteq H$ such that $\nu(N_1)\ge\nu(N_0)$. Since $N_0\cup N_1=H\times H$, $\nu(N_1)\ge\nu(H\times H)/2$ and hence either $\nu(N\times N^c)\ge\nu(H\times H)/4$ or $\nu(N^c\times N)\ge\nu(H\times H)/4$. Without loss of generality we may assume that $\nu(N\times N^c)\ge\nu(H\times H)/4$ as this assumption makes no difference for the following arguments.

For the set $L'\subseteq X$ defined as follows:
\[L'=\bigcup\big\{A\in\mathcal{P}\colon\ f_{\mathcal{P}}(A)\in N,\ g_{\mathcal{P}}(A)\in N^c\big\},\]
we have:
\[\mu(L')=\sum\big\{\mu(A)\colon\ f_{\mathcal{P}}(A)\in N,\ g_{\mathcal{P}}(A)\in N^c,\ A\in\mathcal{P}\big\}=\nu(N\times N^c)\ge\]
\[\ge\nu(H\times H)/4=\frac{1}{4}\sum\big\{\mu(A)\colon f_{\mathcal{P}}(A)\in H,\ g_{\mathcal{P}}(A)\in H,\ A\in\mathcal{P}\big\}.\]
Note that for every $x\in L'$ we have $f(x)\in N$ and $g(x)\in N^c$, so neither $f(x)\in g[L']$, nor $g(x)\in f[L']$. Finally, for the set $L\subseteq X$ given by the formula:
\[L=L'\cup\bigcup\big\{A\in\mathcal{P}\colon\ f_{\mathcal{P}}(A)\not\in H\text{ or }g_{\mathcal{P}}(A)\not\in H\big\},\]
we also have that $f[L]\cap g[L]=\emptyset$ as well as it holds:
\[\mu(L)=\mu(L')+\sum\big\{\mu(A)\colon\ f_{\mathcal{P}}(A)\not\in H\text{ or }g_{\mathcal{P}}(A)\not\in H\big\}\ge\]
\[\ge\frac{1}{4}\sum\big\{\mu(A)\colon f_{\mathcal{P}}(A)\in H,\ g_{\mathcal{P}}(A)\in H\big\}+\frac{1}{4}\sum\big\{\mu(A)\colon\ f_{\mathcal{P}}(A)\not\in H\text{ or }g_{\mathcal{P}}(A)\not\in H\big\}=\]
\[=\frac{1}{4}\sum\big\{\mu(A)\colon\ A\in\mathcal{P}\big\}=\mu(X)/4.\]
\end{proof}

\noindent Note that in the above proposition we may in fact assume that $K$ is finite.

\begin{lem}\label{1/4} Assume that $(X,\Sigma,\mu)$ is a probability space, $K$ is a closed subset of $2^\omega$, and $f,g\colon X \to K$ are measurable and $f\ne g$ $\mu$-almost everywhere. Then, for every $\varepsilon>0$ there is $L\in\Sigma$ such that $\mu(L)>1/4-\varepsilon$ and $f[L] \cap g[L] = \emptyset$.
\end{lem}
\begin{proof}
	For every $s\in 2^{<\omega}$ let $s_0 \in 2^\omega$ be the leftmost branch with root $s$ (think of $s_0$ as of a fixed representative of the clopen $\{x\in 2^\omega\colon s \subseteq x\}$). For $n\in \omega$ define the simple measurable functions $f_n,g_n\colon X \to K$ by putting $f_n(x) = s_0$ and $g_n(x)=t_0$, where $s=f(x)\restriction n$ and $t=g(x)\restriction n$. 

 For every $n\in\omega$ put $A_n =\big\{x\in X\colon f_n(x) = g_n(x)\big\}$ and note that $A_{n+1}\subseteq A_n$. It follows that $\mu\big(A_{n_0}\big)<\varepsilon$ for some $n_0\in\omega$. Indeed, if $\mu(A_n)\ge\varepsilon$ for every $n\in \omega$, then $\mu(\bigcap_n A_n) \geq \varepsilon$, and since $f(x) = g(x)$ for every  $x\in \bigcap_n A_n$, we get a contradiction with the assumption that $f\neq g$ $\mu$-almost everywhere.

Let $X_0 = X \setminus A_{n_0}$, so $\mu(X_0)>1-\varepsilon$ and for every $x\in X_0$ we have $f_n(x)\neq g_n(x)$. By Lemma \ref{malowanie-plotu-new} there is a measurable subset $L\subseteq X_0$ such that $\mu(L)\ge1/4\mu(X_0) > 1/4-\varepsilon$ and $f_n[L] \cap g_n[L] = \emptyset$. It follows that for every $x,x'\in L$ we have $f_n(x)\neq g_n(x')$ and hence $f(x)\neq g(x')$, so $L$ is as desired.
\end{proof}


\begin{prop}\label{14} Let $\mathbb{A}$ be a countable Boolean algebra. Assume that $\varphi\colon \mathbb{A} \to \mathbb{M}_\kappa$ and $\psi\colon \mathbb{A} \to \mathbb{M}_\kappa$ are such homomorphisms that $\bigvee\big\{\varphi(A)\triangle \psi(A)\colon A\in \mathbb{A}\big\} = 1$. Then, for each
	$\varepsilon>0$ there is $C\in
	\mathbb{A}$ such that $\lambda_\kappa\big( \varphi(C) \triangle \psi(C)\big)>1/4-\varepsilon$.
\end{prop}

\begin{proof}
	Let $X$ and $Y$ be the Stone spaces of $\mathbb{M}_\kappa$ and $\mathbb{A}$, respectively. In what follows, we will treat $\lambda_\kappa$ as a Radon measure on $X$. The functions $f_\varphi,f_\psi\colon X\to Y$ are $\lambda_\kappa$-measurable. We claim that $f_\varphi \ne f_\psi$ $\lambda_\kappa$-almost everywhere. If not, then there is a Borel subset $B\subseteq X$ of positive measure such that $f_\varphi\restriction B = f_\psi\restriction B$. By the regularity of $\lambda_\kappa$, there is a compact subset $K$ of $X$ such that $K\subseteq B$ and $\lambda_\kappa(K)>0$. Recall that $\lambda_\kappa$ is a \textit{normal} measure, i.e. for every nowhere dense subset $Z$ of $X$ we have $\lambda_\kappa(Z)=0$, so $K$ cannot be nowhere dense and thus there is $U\in\mathbb{M}_\kappa$ such that $U\subseteq K$ and $\lambda_\kappa(U)>0$. But for every $A\in \mathbb{A}$ we have
$(\varphi(A)\setminus \psi(A)) \cap U=0$ (since otherwise for each $x\in (\varphi(A)\setminus \psi(A)) \cap U$ we would have $f_\varphi(x)\in A$ and $f_\psi(x)\notin A$ and so $f_\varphi(x)\ne f_\psi(x)$ despite the fact that $x\in B$) and $(\psi(A)\setminus \varphi(A)) \cap U=0$ (by a similar argument). This is a contradiction with the assumption that $\bigvee\big\{\varphi(A)\triangle \psi(A)\colon A\in \mathbb{A}\big\} = 1$.

Now fix $\varepsilon>0$ and use Lemma \ref{1/4} to find a measurable subset $L \subseteq X$ such that $\lambda_\kappa(L)>1/4-\varepsilon$ and $f_\varphi[L] \cap f_\psi[L] = \emptyset$ (notice that $Y$ may be treated as a closed subset of $2^\omega$). Using the regularity of $\lambda_\kappa$, we may find a compact subset $M$ of $X$ such that $M\subseteq L$ and $\lambda_\kappa(M)>1/4-\varepsilon$. By the normality of $\lambda_\kappa$, the boundary of $M$ has measure $0$, so there is $V\in\mathbb{M}_\kappa$ such that $V\subseteq M$ and $\lambda_\kappa(V)>1/4-\varepsilon$. Let $C$ be a clopen in $Y$
	separating $f_\varphi[V]$ and $f_\psi[V]$. Then, $V \subseteq \varphi(C)$ but $V \cap \psi(C) = \emptyset$, and hence $\lambda( \varphi(C) \setminus \psi(C))>1/4-\varepsilon$. 
\end{proof}


\begin{proof}[\textbf{Proof of Theorem \ref{distinct}}] 

	Since $\mathbb{M}_\kappa$ is a ccc poset, there are a sequence $(A_n)$ in $\mathbb{A}$ and a maximal antichain $(p_n)$ in $\mathbb{M}_\kappa$ such that $p_n\le\varphi_{\dot{\mathcal{U}}}(A_n)\triangle\varphi_{\dot{\mathcal{V}}}(A_n)$ for every $n\in\omega$. It follows that
\[\bigvee\big\{\varphi_{\dot{\mathcal{U}}}(A_n)\triangle\varphi_{\dot{\mathcal{V}}}(A_n)\colon\ n\in\omega\big\}=1.\]
Let $\mathbb{B}$ be the subalgebra of $\mathbb{A}$ generated by the set $\{A_n\colon\ n\in\omega\}$. By Proposition \ref{14}, there is $C$ in $\mathbb{B}$ (and hence in $\mathbb{A}$) such that $\lambda_\kappa\big(\varphi_{\dot{\mathcal{U}}}(C)\triangle\varphi_{\dot{\mathcal{V}}}(C)\big)>1/4-\varepsilon$. Put $p=\varphi_{\dot{\mathcal{U}}}(C)\triangle\varphi_{\dot{\mathcal{V}}}(C)$ to finish the proof.
\end{proof}

Finally, we are in the position to prove the main theorem of this section.

\begin{proof}[\textbf{Proof of Theorem \ref{uniform-convergence2}}]
%

	Since $(\varphi_n)$ converges uniformly to $\varphi$, there is $m\in\omega$ such that for each $n>m$ and $A\in\mathbb{A}$ we have $\lambda_\kappa\big(\varphi_n(A) \triangle \varphi(A)\big) < 1/5$. We get that there is $m\in\omega$ such that for each
	$n>m$ there are no $p\in\mathbb{M}_\kappa$ and no $A\in\mathbb{A}$ such that $p\le\varphi_n(A) \triangle \varphi(A)$ and $\lambda_\kappa(p)\ge1/5$, which, by Theorem \ref{distinct}, means that for every $n>m$ we have a non-zero $p_n \in \mathbb{M}_\kappa$ such
	that 
	$p_n \Vdash_{\mathbb{M}_\kappa}\dot{\mathcal{U}}_n = \dot{\mathcal{U}}$. 

\end{proof}




Surprisingly, in general we cannot hope to find anything better than the constant $1/4$ in Proposition \ref{malowanie-plotu-new} and in Theorem \ref{distinct}.

\begin{lem}\label{1/4epsilon}
For each non-zero even $n\in\omega$ put $X_n=\{1,2,\ldots,n\cdot(n-1)\}$ and let $\mu_n$ be such that $\mu_n(\{x\})=1/|X_n|$ for each $x\in X_n$. Let $f_n,g_n\colon X_n\to\omega$ be such functions that for every distinct $a,b\in\{1,\ldots,n\}$ there is $x\in X_n$ for which we have $f_n(x)=a$ and $g_n(x)=b$. Then, for each subset $L$ of $X_n$ such that $f_n[L]\cap g_n[L]=\emptyset$ we have:
\[\mu_n(L)\le\frac{1}{4}+\frac{1}{16n-4}.\]
\end{lem}
\begin{proof}
Let $L\subseteq X_n$ be a set such that $f_n[L]\cap g_n[L]=\emptyset$. If $|f_n[L]|=k$ for some $1\le k\le n$, then $|g_n[L]|\le n-k$. Since each pair $(a,b)\in\{1,\ldots,n\}^2$, $a\neq b$, occurs exactly once in the set $f_n[X_n]\times g_n[X_n]$, $L$ may have cardinality at most $k\cdot(n-k)$. The only maximum of the function $\rho_n\colon\mathbb{R}\to\mathbb{R}$ given by the formula $\rho_n(x)=x\cdot(n-x)$ occurs at $x=n/2$ and is equal to $n^2/4$. It follows that the maximal possible cardinality of $L$ is also $n^2/4$ and thus:
\[\mu_n(L)\le\frac{n^2}{4}\cdot\frac{1}{n(n-1)}=\frac{1}{4-\frac{1}{n}}=\frac{1}{4}+\frac{1}{16n-4}.\]
\end{proof}
%
%
%
%
	
\begin{cor} Let $\mathbb{A}$ be an infinite Boolean algebra. Then, for every $\varepsilon>0$ there are $\mathbb{M}_\kappa$-names $\dot{\mathcal{U}}$ and $\dot{\mathcal{V}}$ for ultrafilters on $\mathbb{A}$ having the following properties:
\begin{itemize}
	\item $\Vdash_{\mathbb{M}_\kappa} \dot{\mathcal{U}}\ne \dot{\mathcal{V}}$, and 
	\item for every $p\in\mathbb{M}_\kappa$ for which there is $A\in\mathbb{A}$ such that $p\Vdash A\in\dot{\mathcal{U}}\setminus\dot{\mathcal{V}}$ we have $\lambda_\kappa(p) \leq 1/4+\varepsilon$.
\end{itemize}
\end{cor}

\begin{proof}
	Let $\varepsilon>0$ and fix $n$ such that $\frac{1}{16n-4} < \varepsilon$. Let $(\mathcal{U}_i)_{i<n}$ be sequence of distinct ultrafilters on
$\mathbb{A}$ (in $V$). 

Let $X_n$, $f_n$, $g_n$ be as in Lemma \ref{1/4epsilon}. Fix a maximal antichain $\{P_k\colon k\in X_n\}$ in $\mathbb{M}_\kappa$ consisting of sets of measure $1/|X_n|$. For $i<n$ denote $F_i = \bigvee \{P_k\colon f(k)=i\}$ and $G_i =
\bigvee\{P_k \colon g(k) =i\}$. Notice that both $\{F_i\colon i<n\}$ and $\{G_i\colon i<n\}$ form maximal antichains in $\mathbb{M}_\kappa$ consisting of sets of measure $1/n$.

Now, let $\dot{\mathcal{U}}$ be an $\mathbb{M}_\kappa$-name for an ultrafilter on $\mathbb{A}$ such that $\llbracket \dot{\mathcal{U}} = \mathcal{U}_i \rrbracket = F_i$ for every $i<n$. In other words, $\mathbb{M}_\kappa$ forces that $\dot{\mathcal{U}}$ is one of the
$\mathcal{U}_i$'s and the elements of the antichain $\{F_i\colon i<n\}$ decides which of them. Similarly, choose an $\mathbb{M}_\kappa$-name $\dot{\mathcal{V}}$ in such a way that $\llbracket \dot{\mathcal{V}} = \mathcal{U}_i\rrbracket = G_i$. 
Then, for every $k$ \[ P_k \Vdash \dot{\mathcal{U}} = \mathcal{U}_{f(k)} \mbox{ and } \dot{\mathcal{V}} = \mathcal{U}_{g(k)}, \] 
and since $f(k)\ne g(k)$ we have $\Vdash_{\mathbb{M}_\kappa} \dot{\mathcal{U}} \ne \dot{\mathcal{V}}$.

Let $p\in \mathbb{M}_\kappa$ be such that $\lambda(p)\geq 1/4+\varepsilon$ and let $L = \big\{k\in X_n\colon P_k \wedge p \ne 0\big\}$. Then, $|L| \geq |X_n|(\frac{1}{4}+\varepsilon)$ and so, in terms of Lemma \ref{1/4epsilon},
\[ \mu_n(L)\geq \frac{1}{4} + \varepsilon> \frac{1}{4} + \frac{1}{16n-4}.\]
Thus, by Lemma \ref{1/4epsilon} we have $f[L] \cap g[L] \ne \emptyset$ and so there are $k,l \in L$ such that $f(k) = g(l) = i$ for some $i<n$. Then $P_k \land p$ and $P_l \land p$ are non-zero and
\[ P_k \land p \leq F_i \Vdash \dot{\mathcal{U}} = \mathcal{U}_i \]
but also
\[ P_l \land p \leq G_i \Vdash \dot{\mathcal{V}} = \mathcal{U}_i. \]

Hence, there is no $A\in \mathbb{A}$ such that $p\Vdash A\in \dot{\mathcal{U}}\setminus \dot{\mathcal{V}}$. 
\end{proof}

\begin{rem}\label{rem-malowanie-plotu}
	The core argument of Theorem \ref{distinct} is contained in Proposition \ref{malowanie-plotu-new}, which is a statement in finite combinatorics. It can be formulated in a popular way as	follows. Adam and Eve are painting a picket fence between
	their properties. The fence consists of $n$ many rails. Each rail has two	sides---Adam's and Eve's---and those sides have to be painted in such a way that Adam's side has different colour than Eve's. Proposition \ref{malowanie-plotu-new} says
	that no	matter how many colours they use, there is always a set $B$ of at least $n/4$ many rails with the following property: the set of Adam's colours	used in $B$ is disjoint with the set of Eve's colours used in $B$. (The translation of the
	above into the terms of Proposition \ref{malowanie-plotu-new} is as follows: $X$ is the set of rails, $K$ is the set of colors, $\mu$ is the counting measure, and the functions $f$ and $g$ assign colors to Adam's and Eve's sides of the fence.)
\end{rem}

\section{The Efimov problem}\label{secEfimov}




We now use results from the previous section to obtain a characterization of those Boolean algebras whose Stone spaces do not contain any non-trivial convergent sequences in random extensions.

Recall that a Boolean algebra $\mathbb{A}$ has Seever's \textit{interpolation property} (or, \textit{property (I)}) if for every sequences $(A_n)$ and $(B_n)$ in $\mathbb{A}$ such that $A_n\le B_m$ for every $n,m\in\omega$ there is $C\in\mathbb{A}$ such that $A_n\le C\le B_m$ for every $n,m\in\omega$ (see \cite{Seever}). It is immediate that a Boolean algebra $\mathbb{A}$ has the interpolation property if and only if its Stone space $St(\mathbb{A})$ is an F-space, so e.g. $\mathcal{P}(\omega)/Fin$ has the interpolation property as well as all $\sigma$-complete Boolean algebras do. The result of Dow and Fremlin mentioned in Introduction yields thus the following corollary.

\begin{thm}\cite[Corollary 2.3]{DowFremlin} \label{DowFremlin} 
Let $\mathbb{A}$ be a Boolean algebra in $V$. If $\mathbb{A}$ has the interpolation property in $V$, then in $V^{\mathbb{M}_\kappa}$ the space $\mathrm{St}(\mathbb{A})$ does not have non-trivial convergent sequences.
\end{thm}


Theorem \ref{DowFremlin} together with results obtained above allows us to bring down the question about convergence of ultrafilters in the random extensions to the question about convergence of homomorphisms. Note that if the Stone space $St(\mathbb{A})$ of a Boolean algebra $\mathbb{A}$ contains a non-trivial convergent sequence in $V$, then $St(\mathbb{A})$ contains such a sequence in $V^{\mathbb{M}_\kappa}$, too. The following theorem implies thus e.g. that every countable Boolean algebra admits a sequence of homomorphisms into $\mathbb{M}_\kappa$ which is pointwise algebraically convergent but not uniformly convergent (since $St(\mathbb{A})$ is metrizable; cf. the examples in Sections \ref{ex_alg_not_borel} and \ref{ex_borel_alg_not_unif}).  

\begin{thm}\label{translation} 
Let $\mathbb{A}$ be a Boolean algebra in $V$. The following are equivalent:
	\begin{enumerate}
		\item $\Vdash_{\mathbb{M}_\kappa}$ $St(\mathbb{A})$ does not have non-trivial convergent sequences;
		\item every pointwise algebraically convergent sequence of homomorphisms from $\mathbb{A}$ into $\mathbb{M}_\kappa$ converges  uniformly.
	\end{enumerate}
\end{thm}

\begin{proof} (1)$\implies$(2). Suppose that $(\varphi_n)$ in $\mathcal{H}(\mathbb{A},\mathbb{M}_\kappa)$ converges pointwise algebraically to some homomorphism $\varphi$. Then, by Proposition \ref{convergence}, $\Vdash_{\mathbb{M}_\kappa}$``$\tau_{\varphi_n}$ converges to $\tau_{\varphi}$''. By the assumption, $\Vdash_{\mathbb{M}_\kappa}\forall^\infty n\in\omega\ \tau_{\varphi_n}=\tau_{\varphi}$, and so, by Proposition \ref{uniform-convergence}, $(\varphi_n)$ converges uniformly to $\varphi$.

	(2)$\implies$(1). Suppose now that there exist $p\in \mathbb{M}_\kappa$, and $\mathbb{M}_\kappa$-names $(\dot{\mathcal{U}}_n)$, $\dot{\mathcal{U}}$ for ultrafilters on $\mathbb{A}$ such that $p \Vdash_{\mathbb{M}_\kappa}``(\dot{\mathcal{U}}_n)$
	converges to $\dot{\mathcal{U}}$ and $\forall n \ \dot{\mathcal{U}}_n \ne \dot{\mathcal{U}}$''. Since $\mathbb{M}_\kappa$ is isomorphic to the restricted forcing $p\wedge\mathbb{M}_\kappa=\big\{q\in\mathbb{M}_\kappa\colon q\le p\big\}$ and hence $\mathbb{M}_\kappa$-generic extensions of $V$ are the same as $(p\wedge\mathbb{M}_\kappa)$-generic extensions, we may assume that $p=1$.

	By Proposition \ref{convergence}, the sequence $(\varphi_{\dot{\mathcal{U}}_n})$ converges pointwise algebraically to $\varphi_{\dot{\mathcal{U}}}$ and hence,  by the assumption, uniformly. Theorem \ref{uniform-convergence2} implies that there is
	$n$ and a non-zero $p\in \mathbb{M}_\kappa$ such that $p \Vdash \dot{\mathcal{U}}_n=\dot{\mathcal{U}}$, a contradiction. It follows that there are no non-trivial convergent sequences in $St(\mathbb{A})\cap V^{\mathbb{M}_\kappa}$ .
\end{proof}

Since a compact space containing a copy of $\beta\omega$ must necessarily have weight at least $2^\omega$ and the forcing $\mathbb{M}_\kappa$ preserves cardinals, we immediately get the following corollary. 

\begin{cor}
Assume that $\mathbb{A}\in V$ is a Boolean algebra of size $<\kappa$ and such that in $V$ every pointwise algebraically convergent sequence in $\mathcal{H}(\mathbb{A},\mathbb{M}_\kappa)$ is uniformly convergent. Then, in $V^{\mathbb{M}_\kappa}$, the Stone space of $\mathbb{A}$ is an Efimov space.
\end{cor}

The theorem of Dow and Fremlin yields the following corollary.

\begin{cor}\label{alg_impl_unif} 
If $\mathbb{A}$ has the interpolation property, then every pointwise algebraically convergent sequence of homomorphisms in $\mathcal{H}(\mathbb{A},\mathbb{M}_\kappa)$ converges uniformly.
In particular, this holds if $\mathbb{A}$ is $\sigma$-complete or $\mathbb{A}=\mathcal{P}(\omega)/Fin$.
\end{cor}

Recall that in Section \ref{ex_borel_alg_not_unif} we presented an example of a pointwise algebraic convergent sequence of homomorphisms from the Cantor algebra $\mathbb{C}$ to $\mathbb{M}$ which is not uniformly convergent---this example, as contradicting the conclusion of Corollary \ref {alg_impl_unif}, shows that a Boolean algebra $\mathbb{A}$ in the corollary cannot be to simple.


\section{Open questions and problems}

Fix two Boolean algebras $\mathbb{A}$ and $\mathbb{B}$ and an infinite cardinal number $\kappa$. In this final section we briefly list several open questions and problems concerning the space $\mathcal{H}(\mathbb{A},\mathbb{B})$. 

\begin{que}
What topological properties does the space $\mathcal{H}(\mathbb{A},\mathbb{B})$ have when endowed with the topologies defined in Section \ref{sec-convergence}? In particular, what properties does the space $\mathcal{H}(\mathbb{M}_\kappa,\mathbb{M}_\kappa)$ have?
\end{que}

The following two questions seem natural in the context of Proposition \ref{weakstar_cont} and Corollary \ref{convergences_measures_implications}.

\begin{que}
Is the function $F_\mu$ from Proposition \ref{weakstar_cont} continuous if the spaces $\mathcal{H}(\mathbb{A},\mathbb{B})$ and $P(St(\mathbb{A}))$ are endowed with the pointwise Borel metric topology and weak topology, respectively? 
\end{que}

\begin{prob}
Let $\mathbb{A}$ be a Boolean algebra and $\kappa$ a cardinal. Is there a natural convergence on the space $P(St(\mathbb{M}_\kappa))$  which corresponds to the pointwise algebraic convergence in $\mathcal{H}(\mathbb{A},\mathbb{M}_\kappa)$ in the sense of Corollary \ref{convergences_measures_implications}?
\end{prob}

The following question concerns possible extensions of the diagram from the beginning of Section \ref{sec-examples}.

\begin{que}\label{que-one_conv_another_conv}
What algebraic or structural properties of the Boolean algebras $\mathbb{A}$ and $\mathbb{B}$ imply that a sequence in $\mathcal{H}(\mathbb{A},\mathbb{B})$ convergent in one topology is also convergent in some other one?  
\end{que}

In Section \ref{sec-examples} we have presented several examples of sequences of homomorphisms witnessing that the convergence in one topology may not be sufficient to imply the convergence in some other one. We believe that some of those examples may be generalized to work in more generic situations.

\begin{prob}

Provide a general scheme for constructing a definable sequence of homomorphisms in $\mathcal{H}(\mathbb{A},\mathbb{B})$ such that it is convergent in one topology but not in the other one, or a general scheme for constructing a Borel set or a sequence of Borel sets in $St(\mathbb{A})$ witnessing that a given sequence of homomorphisms in $\mathcal{H}(\mathbb{A},\mathbb{B})$ is not convergent in some topology. 

\end{prob}

By the Dow--Fremlin theorem and Theorem \ref{translation} it follows that if $\mathbb{A}$ has the interpolation property, then every pointwise algebraic sequence in $\mathcal{H}(\mathbb{A},\mathbb{M}_\kappa)$ is uniformly convergent (Corollary \ref{alg_impl_unif}). Since the interpolation property implies the Grothendieck property, we pose the following question being a particular case of Question \ref{que-one_conv_another_conv}.

\begin{que}\label{que-alg_conv_unif}
Assume that $\mathbb{A}$ has the (positive) Grothendieck property. Is every pointwise algebraic convergent sequence in $\mathcal{H}(\mathbb{A},\mathbb{M}_\kappa)$ also uniformly convergent?
\end{que}

Note that, again by Theorem \ref{translation}, obtaining an algebraic or topological proof for an affirmative answer to Question \ref{que-alg_conv_unif} would yield an alternative proof to the theorem of Dow and Fremlin.

\bibliographystyle{alpha}
\bibliography{bib-measures}

\begin{thebibliography}{DPM09}

\bibitem[AK92]{Designs}
Edward~F. Assmus, Jr. and Jennifer~D. Key.
\newblock {\em Designs and their codes}, volume 103 of {\em Cambridge Tracts in
  Mathematics}.
\newblock Cambridge University Press, Cambridge, 1992.

\bibitem[BD19]{BrianDow}
Will Brian and Alan Dow.
\newblock Small cardinals and efimov spaces.
\newblock {\em preprint}, 2019.

\bibitem[BFH99]{BalFraHru}
Bohuslav Balcar, Franti\v{s}ek Fran\v{e}k, and Jan Hru\v{s}ka.
\newblock Exhaustive zero-convergencestructures on {B}oolean algebras.
\newblock {\em Acta Univ. Carolinae, Math. et Phys.}, 40:27--41, 1999.

\bibitem[BGJ98]{Glowczynski}
Bohuslav Balcar, Wies{\l}aw G{\l}\'{o}wczy\'{n}ski, and Thomas Jech.
\newblock The sequential topology on complete {B}oolean algebras.
\newblock {\em Fund. Math.}, 155(1):59--78, 1998.

\bibitem[BNC]{Cegielka}
Piotr Borodulin-Nadzieja and Katarzyna Cegie{\l}ka.
\newblock Homomorphisms, measures and forcing names for ultrafilters.
\newblock {\em preprint}.

\bibitem[DF07]{DowFremlin}
Alan Dow and David Fremlin.
\newblock Compact sets without converging sequences in the random real model.
\newblock {\em Acta Math. Univ. Comenian. (N.S.)}, 76(2):161--171, 2007.

\bibitem[Die84]{Diestel}
Joseph Diestel.
\newblock {\em Sequences and series in {B}anach spaces}, volume~92 of {\em
  Graduate Texts in Mathematics}.
\newblock Springer-Verlag, New York, 1984.

\bibitem[DPM09]{DowPichardo}
Alan Dow and Roberto Pichardo-Mendoza.
\newblock Efimov spaces, {CH}, and simple extensions.
\newblock {\em Topology Proc.}, 33:277--283, 2009.

\bibitem[DS13]{Dow-Shelah}
Alan Dow and Saharon Shelah.
\newblock An {E}fimov space from {M}artin's axiom.
\newblock {\em Houston J. Math.}, 39(4):1423--1435, 2013.

\bibitem[Fed76]{Fedorchuk76}
Vitaly~V. Fedor\v{c}uk.
\newblock Fully closed mappings and the consistency of some theorems of general
  topology with the axioms of set theory.
\newblock {\em Math. USSR Sb.}, 28(1):1--26, 1976.

\bibitem[Ham50]{Hamming}
Richard~W. Hamming.
\newblock Error detecting and error correcting codes.
\newblock {\em Bell System Tech. J.}, 29:147--160, 1950.

\bibitem[Har07]{KPHart}
Klaas~Peter Hart.
\newblock Efimov's problem.
\newblock In {\em Open Problems in Topology II, (E.Pearl, ed.)}, pages
  171--177. Elsevier Science, 2007.

\bibitem[HS20]{HruShi}
Michael Hru\v{s}\'{a}k and Alexander Shibakov.
\newblock Convergent sequences in topological groups.
\newblock {\em preprint}, 2020.

\bibitem[Jec18]{JechSub}
Thomas Jech.
\newblock Convergence and submeasures in {B}oolean algebras.
\newblock {\em Comment. Math. Univ. Carolinae}, 59(4):503--511, 2018.

\bibitem[KN14]{Kopczynski}
Eryk Kopczy\'{n}ski and Damian Niwi\'{n}ski.
\newblock A simple indeterminate infinite game.
\newblock In {\em Logic, computation, hierarchies}, volume~4 of {\em Ontos
  Math. Log.}, pages 205--211. De Gruyter, Berlin, 2014.

\bibitem[Kol95]{Kolmogorov}
Andrey~N. Kolmogorov.
\newblock Complete metric {B}oolean algebras.
\newblock {\em Philos. Stud.}, 77(1):57--66, 1995.
\newblock Translated from the French by Richard Jeffrey.

\bibitem[KS13]{KoszShel}
Piotr Koszmider and Saharon Shelah.
\newblock Independent families in {B}oolean algebras with some separation
  properties.
\newblock {\em Algebra Universalis}, 69(4):305--312, 2013.

\bibitem[See68]{Seever}
Galen~L. Seever.
\newblock Measures on {$F$}-spaces.
\newblock {\em Trans. Amer. Math. Soc.}, 133:267--280, 1968.

\bibitem[SZ19]{Sobota-Zdomsky}
Damian Sobota and Lyubomyr Zdomskyy.
\newblock Convergence of measures in forcing extensions.
\newblock {\em Israel J. Math.}, 232(2):501--529, 2019.

\end{thebibliography}

\end{document}